\documentclass[11pt, a4paper]{article}       

\pdfoutput=1

\usepackage[utf8]{inputenc}
\usepackage[T1]{fontenc}
\usepackage{mathtools, amsthm}
\usepackage{microtype}
\usepackage{amssymb}

\usepackage[colorlinks]{hyperref}

\hypersetup{
  linkcolor=[rgb]{0.3,0.3,0.6},
  citecolor=[rgb]{0.2, 0.6, 0.2},
  urlcolor=[rgb]{0.6, 0.2, 0.2}
}

\usepackage{booktabs}
\usepackage{braket}

\usepackage{tikz}
\usetikzlibrary{positioning}

\newcommand{\CC}{\mathbb{C}}

\DeclareMathOperator{\rank}{R}
\DeclareMathOperator{\borderrank}{\underline{R}}
\DeclareMathOperator{\Tens}{T}
\DeclareMathOperator{\exponent}{\omega}

\DeclarePairedDelimiter\floor{\lfloor}{\rfloor}

\newcommand{\GL}{\mathrm{GL}}
\newcommand{\NN}{\mathbb{N}}
\newcommand{\RR}{\mathbb{R}}
\newcommand{\Oh}{\mathcal{O}}
\newcommand{\tightotimes}{\hspace{-0.05em}\otimes\hspace{-0.05em}}
\newcommand{\GHZ}{\mathrm{GHZ}}
\newcommand{\defin}[1]{\emph{#1}}
\newcommand{\dome}{\textnormal{dome}}

\theoremstyle{plain}
 \newtheorem{theorem}{Theorem}[section]
 \newtheorem*{theorem*}{Theorem}
 \newtheorem{proposition}[theorem]{Proposition}
 \newtheorem{lemma}[theorem]{Lemma}
 \newtheorem{corollary}[theorem]{Corollary}

\theoremstyle{definition}

 \newtheorem*{problem*}{Problem}
 \newtheorem{problem}[theorem]{Problem}
 \newtheorem{remark}[theorem]{Remark}

\begin{document}

\vspace*{1em}
\begin{center}
\Large\textbf{Tensor surgery and tensor rank}\par
\vspace{1em}
\large Matthias Christandl\footnote{QMATH, Department of Mathematical Sciences, University of Copenhagen, Universitetsparken 5, 2100 Copenhagen, Denmark. Email: christandl@math.ku.dk} and Jeroen Zuiddam\footnote{QuSoft, CWI Amsterdam and University of Amsterdam, Science Park 123, 1098 XG Amsterdam, Netherlands. Email: j.zuiddam@cwi.nl}\par
\end{center}
\vspace{0.5em}

\begin{abstract}
We introduce a method for transforming low-order tensors into higher-order tensors and apply it to tensors defined by graphs and hypergraphs. The transformation proceeds according to a surgery-like procedure that splits vertices, creates and absorbs virtual edges and inserts new vertices and edges. We show that tensor surgery is capable of preserving the low rank structure of an initial tensor decomposition and thus allows to prove nontrivial upper bounds on tensor rank, border rank and asymptotic rank of the final tensors. We illustrate our method with a number of examples. Tensor surgery on the triangle graph, which corresponds to the matrix multiplication tensor, leads to nontrivial rank upper bounds for all odd cycle graphs, which correspond to the tensors of iterated matrix multiplication. In the asymptotic setting we obtain upper bounds in terms of the matrix multiplication exponent $\omega$ and the rectangular matrix multiplication parameter $\alpha$. These bounds are optimal if $\omega$ equals two. We also give examples that illustrate that tensor surgery on general graphs might involve the absorption of virtual hyperedges and we provide an example of tensor surgery on a hypergraph. Besides its relevance in algebraic complexity theory, our work has applications in quantum information theory and communication complexity. 
\end{abstract}


\section{Introduction}
This paper introduces a method for proving upper bounds on tensor rank, border rank and asymptotic tensor rank. The method gives particularly clean results when applied to tensors that are defined combinatorially. Let us first illustrate the combinatorial description that we are using and illustrate the method.
\subsection{Illustration}
The most famous example of a tensor that fits into our combinatorial framework (and which plays an important role in this paper) is the two-by-two matrix multiplication tensor, which is the 3-tensor described by the triangle graph $C_3$
\begin{alignat*}{2}
&\Tens_2\Bigl(\begin{minipage}{1.1cm}
\begin{center}
\begin{tikzpicture}[vertex/.style = {circle, fill, black, minimum width = 1.mm, inner sep=0pt}]
    \path[coordinate] (0,0)  coordinate(A)
                ++( 2*1*60+30:0.8cm) coordinate(B)
                ++( 2*2*60+30:0.8cm) coordinate(C);
	\draw [line width=0.2mm]  (A) node [vertex] {} -- (B) node [vertex] {} -- (C) node [vertex] {} -- (A);
\end{tikzpicture}
\end{center}
\end{minipage}\Bigr) \,&=& \sum_{i\in\{0,1\}^3}\hspace{-1ex} (b_{i_1}\tightotimes b_{i_2})\otimes (b_{i_2} \tightotimes b_{i_3})\otimes (b_{i_3} \tightotimes b_{i_1}) \,\in\, (\CC^2\tightotimes\CC^2)^{\otimes 3},
\intertext{where $\{b_0, b_1\}$ is the standard basis of $\CC^2$.
For a space $\CC^{n_1} \otimes \cdots \otimes \CC^{n_k}$ of $k$-tensors we refer to the $\CC^{n_i}$ as the tensor legs.
 Informally, the graph--tensor correspondence is as follows: each vertex of the graph corresponds to a tensor leg and each edge in the graph corresponds to an index to sum over, shared between tensor legs (see Section~\ref{prelim} for a formal definition). By default we view the above tensor as a 3-tensor, but we will sometimes view it as a 6-tensor. 
Another important example is the so-called rank-two unit 3-tensor, which corresponds to the hypergraph on three vertices with a single hyperedge~$\{1,2,3\}$}
&\Tens_2\Bigl(\begin{minipage}{1.1cm}
\begin{center}
\begin{tikzpicture}[vertex/.style = {circle, fill, black, minimum width = 1.mm, inner sep=0pt}]
    \path[coordinate] (0,0)  coordinate(A)
                ++( 2*1*60+30:0.8cm) coordinate(B)
                ++( 2*2*60+30:0.8cm) coordinate(C);
	\draw [line width=0.2mm, rounded corners, fill=black!10]  (A) node [vertex] {} -- (B) node [vertex] {} -- (C) node [vertex] {} -- cycle;
\end{tikzpicture}
\end{center}
\end{minipage}\Bigr) 
\,&=&\sum_{i\in\{0,1\}}\hspace{-1ex} b_{i}\otimes b_{i}\otimes b_{i} \,\in\, \CC^2\otimes \CC^2\otimes \CC^2.
\intertext{In the algebraic complexity theory literature, the two-by-two matrix multiplication tensor is usually denoted by $\langle 2,2,2\rangle$ and the rank-two unit 3-tensor by~$\langle 2 \rangle$. As a final illustrative example consider the complete graph on 4 vertices~$K_4$ and the corresponding 4-tensor}
&\Tens_2\Bigl(\begin{minipage}{1.3cm}
\begin{center}
\begin{tikzpicture}[vertex/.style = {circle, fill, black, minimum width = 1.mm, inner sep=0pt}, novertex/.style = {minimum width = 0.mm, inner sep=0pt}]
	\node[vertex] (0,0) (X) {};
	\node[vertex, below= 3mm of X] (Y) {};
	\node[vertex, below right= 4mm of Y, xshift=1mm, yshift=0.5mm] (Z) {};
	\node[vertex, below left= 4mm of Y, xshift=-1mm, yshift=0.5mm] (V) {};
	\draw[line width=0.2mm] (X) --  (Y);
	\draw[line width=0.2mm] (X) edge[bend left=10]  (Z);
	\draw[line width=0.2mm] (X) edge[bend right=10] (V);
	\draw[line width=0.2mm] (Y) -- (Z) -- (V) -- (Y);
\end{tikzpicture}
\end{center}
\end{minipage} \Bigr)  
\,&=& \hspace{-0.5em}\sum_{i\in \{0,1\}^6} \hspace{-0.5em}(b_{i_1}\tightotimes b_{i_2}\tightotimes b_{i_3}) \otimes (b_{i_3}\tightotimes b_{i_4} \tightotimes b_{i_5}) \otimes (b_{i_2}\tightotimes b_{i_4} \tightotimes b_{i_6}) \\[-1em]
&&&\hspace{3.5em}  \otimes (b_{i_1}\tightotimes b_{i_5} \tightotimes b_{i_6})\,\in\, (\CC^2\tightotimes\CC^2\tightotimes \CC^2)^{\otimes 4}.
\end{alignat*}

Our aim is to prove nontrivial upper bounds on tensor rank, border rank and asymptotic tensor rank. Let us for now focus on tensor rank. The tensor rank of a $k$-tensor in $\CC^{n_1} \otimes \cdots \otimes \CC^{n_k}$ is the smallest number~$r$ such that the tensor can be written as a sum of $r$ simple tensors $v_1\otimes \cdots \otimes v_k$ with $v_i \in \CC^{n_i}$. The tensor rank of a tensor $t$ is denoted by $\rank(t)$. Since tensor rank is invariant under the action of the group $\GL_{n_1}\times \cdots \times \GL_{n_k}$ we will identify tensors that are in the same orbit under this group action. 

Going back to our examples, a nontrivial upper bound of seven on the tensor rank of $\langle 2,2,2\rangle$ was obtained by Strassen by constructing an efficient bilinear algorithm for multiplying two-by-two matrices \cite{strassen1969gaussian}, a breakthrough result in algebraic complexity theory. The second tensor $\langle 2 \rangle$ is the canonical example of a tensor of rank two. For the third tensor, observe that the graph contains a triangle, and hence Strassen's decomposition of $\langle 2,2,2\rangle$ can directly be upgraded to a nontrivial decomposition of this tensor of size 56. 
This direct upgrading idea does not work when a tensor corresponds to a graph without triangles, say the five-cycle. The method that we will describe below allows us to prove nontrivial rank upper bounds even for tensors corresponding to graphs that do not contain triangles.    


Tensor surgery goes as follows. The central idea is to transform a good decomposition of a well-chosen starting tensor into a good decomposition of a goal tensor.
%
Take a tensor~$t$ of which we know (an upper bound on) the tensor rank (or border rank, or asymptotic rank). Then, linearly split up a tensor leg of $t$ into multiple tensor legs and take the tensor product with another tensor $s$ (``inserting~$s$'') to obtain our goal tensor, carefully keeping track of the increase in rank that this combined operation causes. 
Combining our knowledge of the decomposition of $t$ with our knowledge of the rank increase gives a decomposition of the goal tensor.

We illustrate tensor surgery with the 5-tensor of the five-cycle $C_5$,
\begin{equation}\label{eqfive}
\Tens_2\Bigl(\begin{minipage}{1.0cm}
\begin{center}
\begin{tikzpicture}[vertex/.style = {circle, fill, black, minimum width = 1.mm, inner sep=0pt}]
    \path[coordinate] (0,0)  coordinate(A)
                ++( 2*1*180/5-18:0.5cm) coordinate(B1)
                ++( 2*2*180/5-18:0.5cm) coordinate(B)
                ++( 2*3*180/5-18:0.5cm) coordinate(C)
                ++( 2*4*180/5-18:0.5cm) coordinate(D);
	\draw [line width=0.2mm] (D) node [vertex] {}  -- (A) node [vertex] {} -- (B1) node [vertex] {} -- (B) node [vertex] {} -- (C) node [vertex] {} -- (D);
\end{tikzpicture}
\end{center}
\end{minipage}\Bigr) \,= \,\sum_{i\in\{0,1\}^5}\hspace{0ex} b_{i_1i_2}\otimes b_{i_2i_3}\otimes b_{i_3i_4} \otimes b_{i_4i_5} \otimes b_{i_5 i_1} \in (\CC^2\tightotimes \CC^2)^{\otimes 5},
\end{equation}
where $b_{ij} \coloneqq b_i \tightotimes b_j$ with $\{b_0, b_1\}$ the standard basis of $\CC^2$.
The defining decomposition~\eqref{eqfive} of $\Tens_2(C_5)$ has size 32. We can improve this rank upper bound as follows. 
Define the linear map $\phi$ by
\begin{alignat*}{2}
\phi{}:{} &\CC^{2}\otimes \CC^{2} \to (\CC^2 \otimes \CC^2)^{\otimes 3} \\
&u \otimes v\mapsto \hspace{-1ex}\sum_{j\in\{0,1\}^2}\hspace{-1ex} (u \tightotimes b_{j_1}) \otimes (b_{j_1} \tightotimes b_{j_2}) \otimes (b_{j_2}\tightotimes v).
\end{alignat*}
Let $\psi : (\CC^2\tightotimes \CC^2)^{\otimes 3} \to  (\CC^2\tightotimes \CC^2)^{\otimes 5}$ be the map that applies $\phi$ at the first tensor leg. 
 Then
\[
\Tens_2(C_5) = \psi(\Tens_2(C_3)).
\]
For $\Tens_2(C_3)$ we have a good decomposition, namely Strassen's decomposition.
Define the elements $b_{\textsf{+}} \coloneqq b_0 + b_1$ and $b_{\textsf{--}} \coloneqq b_0 - b_1$ in~$\CC^2$. For any pair of symbols $x,y \in \{0,1,\textsf{+},\textsf{--}\}$ define $b_{xy} \coloneqq b_x \otimes b_y \in \CC^2 \otimes \CC^2$. Strassen's decomposition is
\begin{alignat*}{1}
\Tens_2(C_3)\,\, =\quad &-\,\, b_{\textsf{--}0}\otimes b_{0\textsf{+}} \otimes b_{11}\quad -\,\, b_{11}\otimes b_{\textsf{--}0} \otimes b_{0\textsf{+}}\quad -\,\, b_{0\textsf{+}}\otimes b_{11}\otimes b_{\textsf{--}0}\\[0.3em]
&+\,\, b_{\textsf{--}1}\otimes b_{1\textsf{+}}\otimes b_{00}\quad +\,\, b_{00}\otimes b_{\textsf{--}1}\otimes b_{1\textsf{+}}\quad +\,\, b_{1\textsf{+}}\otimes b_{00}\otimes b_{\textsf{--}1}\\[0.3em]
&+\,\,(b_{00}+b_{11})\otimes (b_{00}+b_{11})\otimes (b_{00}+b_{11}).
\end{alignat*}
Applying the linear map $\psi$ to the decomposition yields
\newcommand{\strleft}{\,}
\newcommand{\strright}{\,}
\begin{alignat*}{1}
\Tens_2(C_5) ={}&\psi(\Tens_2(C_3))\\ 
={} &-\strright \phi(b_{\textsf{--}0})\otimes b_{0\textsf{+}} \otimes b_{11}\strleft -\strright \phi(b_{11})\otimes b_{\textsf{--}0} \otimes b_{0\textsf{+}}\strleft -\strright \phi(b_{0\textsf{+}})\otimes b_{11}\otimes b_{\textsf{--}0}\\[0.3em]
&+\strright \phi(b_{\textsf{--}1})\otimes b_{1\textsf{+}}\otimes b_{00}\strleft +\strright \phi(b_{00})\otimes b_{\textsf{--}1}\otimes b_{1\textsf{+}}\strleft +\strright \phi(b_{1\textsf{+}})\otimes b_{00}\otimes b_{\textsf{--}1}\\[0.3em]
&+\strright\phi(b_{00}+b_{11})\otimes (b_{00}+b_{11})\otimes (b_{00}+b_{11}).
\end{alignat*}
For any $x,y\in \{0,1,\textsf{+}, \textsf{--}\}$, $\phi(b_{xy}) = \sum_{j\in \{0,1\}^2} (b_x \tightotimes b_{j_1}) \otimes (b_{j_1} \tightotimes b_{j_2}) \otimes (b_{j_2} \tightotimes b_y)$, which has rank 4 as a 3-tensor. We have $\phi(b_{00} + b_{11})=\sum_{i\in \{0,1\}^3}  (b_{i_1} \tightotimes b_{i_2}) \otimes (b_{i_2} \tightotimes b_{i_3}) \otimes (b_{i_3} \tightotimes b_{i_1})$ for the remaining term, which equals $\Tens_2(C_3)$ and thus has rank 7 as a 3-tensor, invoking Strassen's decomposition for the second time. Therefore, $\rank(\Tens_2(C_5)) \leq 6\cdot 4 + 1 \cdot 7 = 31$.
This means that we have achieved our goal of constructing a nontrivial rank-31 decomposition of the goal tensor~$\Tens_2(C_5)$,  smaller than the trivial decomposition of size 32. 

Identifying $\Tens_2(C_3)$ with the graph $C_3$, we think of $\psi$ as a ``surgery map'' that splits a vertex into two vertices and inserts a new vertex together with two edges.
In pictures, the effect of applying $\psi$ is
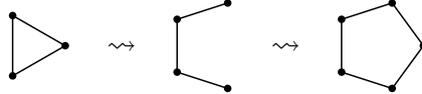
\[
\begin{minipage}{0.8cm}
\begin{tikzpicture}[vertex/.style = {circle, fill, black, minimum width = 1.mm, inner sep=0pt}]
    \path[coordinate] (0,0)  coordinate(A)
                ++( 2*1*60+30:0.8cm) coordinate(B)
                ++( 2*2*60+30:0.8cm) coordinate(C);
	\draw [line width=0.2mm]  (A) node [vertex] {} -- (B) node [vertex] {} -- (C) node [vertex] {} -- (A);
\end{tikzpicture}
\end{minipage}
\quad\leadsto\quad
\begin{minipage}{0.8cm}
\begin{tikzpicture}[vertex/.style = {circle, fill, black, minimum width = 1.mm, inner sep=0pt}]
    \path[coordinate] (0,0)  coordinate(A)
                ++( 2*1*180/5-18:0.7cm) coordinate(B1)
                ++( 2*2*180/5-18:0.7cm) coordinate(B)
                ++( 2*3*180/5-18:0.7cm) coordinate(C)
                ++( 2*4*180/5-18:0.7cm) coordinate(D);
	\draw [line width=0.2mm] (D) node [vertex] {}  -- (A) node [vertex] {}  (B) node [vertex] {} -- (C) node [vertex] {} -- (D);

\end{tikzpicture}
\end{minipage}
\quad\leadsto\quad
\begin{minipage}{0.8cm}
\begin{tikzpicture}[vertex/.style = {circle, fill, black, minimum width = 1.mm, inner sep=0pt}]
    \path[coordinate] (0,0)  coordinate(A)
                ++( 2*1*180/5-18:0.7cm) coordinate(B1)
                ++( 2*2*180/5-18:0.7cm) coordinate(B)
                ++( 2*3*180/5-18:0.7cm) coordinate(C)
                ++( 2*4*180/5-18:0.7cm) coordinate(D);
	\draw [line width=0.2mm] (D) node [vertex] {}  -- (A) node [vertex] {} -- (B1) node [vertex] {} -- (B) node [vertex] {} -- (C) node [vertex] {} -- (D);
\end{tikzpicture}
\end{minipage}
\]
Splitting a tensor leg possibly increases tensor rank, as is the case with the term $b_{00} + b_{11}$ above. To remind us of this, we like to decorate the picture with a ``virtual edge'' connecting the cut vertices,
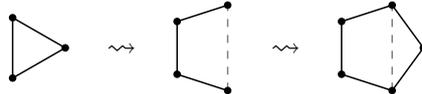
\[
\begin{minipage}{0.8cm}
\begin{tikzpicture}[vertex/.style = {circle, fill, black, minimum width = 1.mm, inner sep=0pt}]
    \path[coordinate] (0,0)  coordinate(A)
                ++( 2*1*60+30:0.8cm) coordinate(B)
                ++( 2*2*60+30:0.8cm) coordinate(C);
	\draw [line width=0.2mm]  (A) node [vertex] {} -- (B) node [vertex] {} -- (C) node [vertex] {} -- (A);
\end{tikzpicture}
\end{minipage}
\quad\leadsto\quad
\begin{minipage}{0.8cm}
\begin{tikzpicture}[vertex/.style = {circle, fill, black, minimum width = 1.mm, inner sep=0pt}]
    \path[coordinate] (0,0)  coordinate(A)
                ++( 2*1*180/5-18:0.7cm) coordinate(B1)
                ++( 2*2*180/5-18:0.7cm) coordinate(B)
                ++( 2*3*180/5-18:0.7cm) coordinate(C)
                ++( 2*4*180/5-18:0.7cm) coordinate(D);
	\draw [line width=0.2mm, dashed, gray] (A) -- (B);
	\draw [line width=0.2mm] (D) node [vertex] {}  -- (A) node [vertex] {}  (B) node [vertex] {} -- (C) node [vertex] {} -- (D);

\end{tikzpicture}
\end{minipage}
\quad\leadsto\quad
\begin{minipage}{0.8cm}
\begin{tikzpicture}[vertex/.style = {circle, fill, black, minimum width = 1.mm, inner sep=0pt}]
    \path[coordinate] (0,0)  coordinate(A)
                ++( 2*1*180/5-18:0.7cm) coordinate(B1)
                ++( 2*2*180/5-18:0.7cm) coordinate(B)
                ++( 2*3*180/5-18:0.7cm) coordinate(C)
                ++( 2*4*180/5-18:0.7cm) coordinate(D);
	\draw [line width=0.2mm, dashed, gray] (A) -- (B);
	\draw [line width=0.2mm] (D) node [vertex] {}  -- (A) node [vertex] {} -- (B1) node [vertex] {}  -- (B) node [vertex] {} -- (C) node [vertex] {} -- (D);
\end{tikzpicture}
\end{minipage}
\]
The crux is the triangle appearing on the right. This triangle indicates the worst-case situation where $\phi(b_{00} + b_{11}) = \Tens_2(C_3)$ has rank 7. Of course, to get a good decomposition it is important to also keep track of the best-case situation where $\phi(b_{xy})$ has rank 4.   

\subsection{Main results}
Let $G = (V,E)$ be a graph with vertex set $V$ and edge set $E$, and let $n$ be a natural number. Let $b_1,\ldots, b_n$ be the standard basis of $\CC^n$.  We define the order-$|V|$ tensor $\Tens_n(G)$ as
\[
\Tens_n(G) \coloneqq \sum_{i\in [n]^E} \bigotimes_{v \in V} \Bigl(\bigotimes_{\substack{e\in E:\\ v\in e}} b_{i_e}\Bigr),
\]
summing over all tuples $i$ indexed by $E$ with entries in $[n]\coloneqq\{1,2,\ldots,n\}$.
Let $\rank(\Tens_n(G))$ be the tensor rank of the tensor $\Tens_n(G)$ and let
\[
\omega(\Tens_2(G))\coloneqq \lim_{n\to\infty} \tfrac1n \log_2 \rank(\Tens_2(G)^{\otimes n}) = \lim_{n\to\infty} \log_{n} \rank(\Tens_{n}(G))
\]
be the \emph{exponent} of $\Tens_2(G)$, a measure of the asymptotic behaviour of the tensor rank of $\Tens_n(G)$. (The limit exists and equals the infimum by Fekete's lemma. The equality follows by relating $\log_{2^n} \rank(T_{2^n}(G))$ to $\log_n \rank(T_n(G))$.)

The triangle tensor $\Tens_n(C_3)$ is well-studied, because its tensor rank equals the number of bilinear scalar multiplications required to multiply two $n\times n$ matrices. Moreover the so-called \emph{matrix multiplication exponent}
\[
\omega \coloneqq \omega(\Tens_2(C_3)) 
\]
equals the smallest number $\beta \in \RR$ such that for any $\varepsilon>0$ two $n\times n$ matrices can be multiplied with $\Oh(n^{\beta + \varepsilon})$ scalar multiplications and additions. 
A priori, $2 \leq \omega\leq 3$. As mentioned above, Strassen showed that $\rank(\Tens_2(C_3)) \leq 7$ by constructing an efficient bilinear algorithm for multiplying two-by-two matrices, thereby showing that $\omega \leq \log_2 7$. The rank upper bound was later proven to be tight by Winograd \cite{winograd1971multiplication}. Since Strassen's breakthrough, much effort has been put into obtaining better bounds on $\omega$, the state of the art being $2\leq \omega < 2.3728639$ \cite{le2014powers}. Proposed approaches towards obtaining lower bounds on the rank of $\Tens_n(C_3)$ and $\omega$ include Strassen's asymptotic spectra \cite{MR1089800,MR929980}, the geometric complexity theory programme \cite{MR2138544,MR2932001} and Young flattenings \cite{MR709378,landsberg2011new,MR3081636}. The recent best upper bounds on $\omega$ have been obtained by extending a construction of Coppersmith and Winograd \cite{coppersmith1987matrix}. It was shown however that this type of extension cannot prove an upper bound on $\omega$ below $2.3078$ \cite{MR3388238}. Recently, good upper bounds (not the best) have been obtained by a group-theoretic approach which does not fall under this type of extensions~\cite{MR3202968}.

In this paper we go into unexplored terrain by studying the tensor rank and exponent of larger cycle tensors.
Our first result is that, for any odd $k$, the tensor of the $k$-cycle has a nontrivial tensor rank.\footnote{For even $k$, trivially $\rank(\Tens_2(C_k)) = 2^k$, while for odd $k$, trivially $2^{k-1} \leq \rank(\Tens_2(C_k)) \leq 2^k$, see Section~\ref{prelim}.}
\begin{theorem*}
Let $k$ be odd. Then $\rank(\Tens_2(C_k)) \leq 2^k - 1$.
\end{theorem*}
This was previously only known for odd $k\leq 5$ \cite{buhrman2016nondeterministic}.
Let $\omega_k \coloneqq \omega(\Tens_2(C_k))$. We prove a relationship between the exponents of odd cycles.
\begin{theorem*}
Let $k, \ell$ be odd. Then $\omega_{k+\ell - 1} \leq \omega_{k} + \omega_{\ell}$.
\end{theorem*}
We moreover prove an upper bound on the exponent of odd cycles in terms of the dual exponent of matrix multiplication $\alpha$, which we will define in Section~\ref{prelim}.
\begin{theorem*} Let $k$ be odd. Then
\[
\omega_k \leq k - \alpha\Bigl(1 + \frac{1 - \alpha}{k - 1 + \alpha}\Bigr) \leq k - \alpha.
\]
\end{theorem*}
In particular, since $0.3029805 < \alpha \leq 1$, the exponent $\omega_k$ is bounded away from~$k$ by a constant.

Our results on the exponent of odd cycles are optimal in the sense that if $\omega = 2$, then $\omega_k = k-1$ for all odd $k$. For tensor rank, many open problems remain. As a concrete example, we do not know the value of $\rank(\Tens_2(C_5))$. We know that it is at least 25 and at most 31 (see Remark\nobreakspace \ref {openquestion}).

Besides looking at graphs, we will in this paper explore tensor surgery on hypergraphs, where one splits up a tensor leg into multiple tensor legs and instead of a graph inserts a hypergraph. We derive a number of results on the asymptotic rank similar to the ones on cycles graphs.

As the main results indicate, tensor surgery works well for sparse graphs. In a subsequent paper, with an entirely different method, we have obtained nontrivial upper bounds on the exponent of dense graphs \cite{christandl2016asymptotic} (which in turn can be used again as starting tensors for the tensor surgery put forward in this work). The common theme of that paper and the current paper is the following open problem that generalizes the problem of computing the matrix multiplication exponent $\omega$.
\begin{problem}
Let $G$ be a graph. What is the value of $\omega(\Tens_2(G))$?
\end{problem}

%
%

\subsection{Connections to other work}
Tensor rank has been studied in various fields other than algebraic complexity theory: in algebraic statistics \cite{pachter2005algebraic}, in signal processing \cite{comon1996decomposition}, in algebraic geometry in the context of $r$th secant varieties of Segre varieties \cite{landsberg2012tensors}, 
in quantum information theory as a monotone for stochastic local operations and classical communication (SLOCC) \cite{chen2010tensor,vrana2015asymptotic}, and in communication complexity \cite{draisma2009partition,buhrman2016nondeterministic} to characterize the complexity of communication problems, to name a few. 

The tensor rank of graph tensors in particular has the following applications.
In quantum information language, for any graph $G$ the tensor $\Tens_n(G)$ is the (unnormalized) quantum state obtained by identifying the vertices of~$G$ with quantum systems and letting each edge of the graph correspond to a dimension-$n$ Einstein--Podolsky--Rosen (EPR) pair shared among the vertices contained in the edge. For example, if $G$ contains just a single edge, then $\Tens_n(G)$ is the EPR pair $\sum_{i=1}^n b_i \otimes b_i$. The notion of a graph tensor naturally extends to hypergraphs. For any hypergraph $H$ consisting of a single edge $\{1,2,\ldots, k\}$, the tensor $\Tens_n(H)$ is the (unnormalized) Greenberger--Horne--Zeilinger (GHZ) state $\sum_{i=1}^n \ket{i}^{\otimes k}$ of rank $n$ and order $k$. 
Let $G$ be a graph on $k$ vertices. The tensor rank $\rank(\Tens_n(G))$ is the smallest number $r$ such that $\Tens_r(H)$ can be transformed into $\Tens_n(G)$ under stochastic local operations and classical communication (SLOCC). The exponent $\exponent(\Tens_n(G))$ is the smallest real number $\beta$ such that $\lceil\beta + o(n)\rceil$ copies of $\Tens_2(H)$ can be transformed into $n$ copies of~$\Tens_2(G)$ by SLOCC, when $n$ goes to infinity.

In communication complexity, a notion related to tensor rank called support rank characterizes the so-called nondeterministic quantum communication complexity with quantum broadcast communication of any boolean function \cite{buhrman2016nondeterministic}. Graph tensors correspond to the graphwise equality problem, so our upper bounds can be interpreted as upper bounds on the complexity of certain graphwise equality problems. Surprisingly, there is an implication in the other direction, namely upper bounds on the support rank of $\Tens_n(C_3)$ imply slightly worse upper bounds on the tensor rank of $\Tens_n(C_3)$. More precisely, one can define a support rank exponent $\omega_s$ analogous to the exponent $\omega$ and then the inequality $\omega \leq \tfrac32 \omega_s - 1$ holds \cite{MR3202968}. This tightly connects the study of asymptotic rank to communication complexity and in part motivated the present work. 
%
In \cite{buhrman2016nondeterministic}, an explicit size-31 decomposition of $\Tens_2(C_5)$ was found with computer assistance and nontrivial asymptotic upper bounds were given for all odd~$k$.
More precisely, Strassen's laser method combined with the distillation result of \cite{vrana} (see Equation \eqref{distillation2} in this paper) was used to get the upper bound $\omega(\Tens_2(C_k)) \leq \min_{q\geq2} \log_q ((q+1)^k/4)$. This bound converges to~$k$ when $k$ goes to infinity. The present paper thus answers in the positive the open question of whether $\omega_k$ is uniformly bounded away from $k$.




\subsection{Outline of the paper}

We will begin by discussing some preliminaries in Section~\ref{prelim}.
In Section\nobreakspace \ref {cycles} we prove rank upper bounds and exponent upper bounds for odd cycles. 
In Section\nobreakspace \ref {hyper} we explore the more general hypergraph variant of tensor surgery. 

\section{Preliminaries}\label{prelim}

After a formal definition of graph tensors, this section discusses basic notions and results around tensor rank, border rank and asymptotic rank. The sections concludes with a discussion of the lower bounds methods of flattening and Young flattening.

\subsection{Graph tensors}
All our vector spaces will be complex finite-dimensional vector spaces. However, the ideas in this paper will work over any field.
Let $G = (V,E)$ be a graph and let $n$ be a natural number. Let $b_1,\ldots, b_n$ be the standard basis of~$\CC^n$.  We define the $|V|$-tensor $\Tens_n(G)$ as
\[
\Tens_n(G) \coloneqq \sum_{i\in [n]^E} \bigotimes_{v \in V} \Bigl(\bigotimes_{\substack{e\in E:\\ v\in e}} b_{i_e}\Bigr),
\]
where the sum is over all tuples $i$ indexed by $E$ with entries in the set $[n]\coloneqq\{1,2,\ldots,n\}$. Equivalently, we can define $\Tens_n(G)$ as follows:
\[
\Tens_n(G) = \bigotimes_{e \in E} \sum_{i\in [n]}  (b_i \otimes b_i)_e \otimes (1\otimes \cdots \otimes 1)_{V\setminus e}.
\]
Here the subscripts $e$ and $V\setminus\{e\}$ in a summand indicate that the tensor legs of the summand are permuted by $(1, e_1)(2, e_2)$, and the large tensor product is a tensor product of $|V|$-tensors.
We write~$\Tens$ for $\Tens_2$.
We can safely ignore the fact that this tensor depends on the choice of order of the edges and vertices, since the tensor rank does not depend on this order, and we identify tensors that are equivalent up to local $\GL$-action. This definition directly extends to hypergraphs. These tensors were studied in \cite{vrana2015asymptotic,vrana} with the notation $\GHZ_n^G = \Tens_n(G)$. Note that $\Tens_2(G)^{\otimes k} = \Tens_2(G^{\cup k}) = \Tens_{2^k}(G)$ where $G^{\cup k}$ denotes the multigraph obtained from $G$ by taking the union of $k$ copies of~$G$ on the same vertex set.

\subsection{Tensor rank and exponent}
The \defin{tensor rank} of a $k$-tensor in $\CC^{n_1} \otimes \cdots \otimes \CC^{n_k}$ is the smallest number~$r$ such that the tensor can be written as a sum of $r$ simple tensors $v_1\otimes \cdots \otimes v_k$ with $v_i \in \CC^{n_i}$. The tensor rank of a tensor $t$ is denoted by $\rank(t)$.
%
%
When~$k$ equals~2, tensor rank is the same as matrix rank and is thus efficiently computable. When $k$ is at least 3, however, deciding tensor rank is NP-hard \cite{haastad1990tensor}, see also \cite{shitov2016hard} and \cite{schaefer2016complexity} for recent developments. The border rank of $t$ is the smallest number $r$ such that $t$ can be approximated by tensors of rank at most $r$ in the Euclidean topology. We denote border rank by $\borderrank(t)$. We refer to \cite{burgisser1997algebraic,landsberg2012tensors} for an introduction to tensor rank and border rank. We mention in particular that there is an algebraic version of border rank which is also defined over finite fields.

For two $k$-tensors $s \in U_1\otimes \cdots \otimes U_k$ and $t\in V_1 \otimes \cdots \otimes V_k$ we say \defin{$s$ restricts to $t$}, and write $s \geq t$, if there exist linear maps $A_i: U_i\to V_i$ such that $(A_1\otimes \cdots \otimes A_k) s = t$. Define the asymptotic conversion rate from $s$ to $t$ as
\[
\omega(s, t) \coloneqq \lim_{n\to\infty} \frac{1}{n} \min\{m \in \NN \mid s^{\otimes m} \geq t^{\otimes n} \}.
\]
The minimum of the empty set is defined to be $\infty$. The limit exists and equals the supremum over $n$, see Lemma 1.1 in \cite{MR929980}.
%
Let $[k]$ denote the hypergraph with vertex set $[k]$ and a single edge containing all vertices. We define the \defin{rank-$n$ unit $k$-tensor} $\Tens_n(k)$ as 
\[
\Tens_n(k)\coloneqq\Tens_n([k]) = \sum_{i\in[n]} b_i^{\otimes k}.
\]
(So, $\Tens_2(3) = \langle 2 \rangle$.)
The asymptotic log-rank or \defin{exponent} of a tensor $t$ is defined as the limit 
\begin{equation}\label{omega}
\omega(t) \coloneqq \omega(\Tens(k), t) = \lim_{n\to\infty} \frac{1}{n} \min \{m\in \NN \mid 2^m \geq \rank(t^{\otimes n})\}.
\end{equation}
%
%
The parameter $\omega(t)$ thus measures how many copies of $\Tens(k)$ are asymptotically needed to create a copy of $t$ by restriction. On the other hand, the parameter $\omega(s, \Tens(k))^{-1}$ measures how many copies of $\Tens(k)$ can asymptotically be created from one copy of $s$ by restriction. We call this the subexponent of $s$.

For any $k\in \NN$, let $C_k$ be the cycle graph with vertex set $[k]$ and edge set $\{\{1,2\},$ $\{2,3\},$ $\ldots,$ $\{k, 1\}\}$.
 A well-known result is that, asymptotically, two copies of~$\Tens(3)$ can be obtained from the triangle tensor  \cite{strassen1987relative}:  
\begin{align}
\omega(\Tens(C_3), \Tens(3))^{-1} &= 2.\label{distillation}\\
\intertext{It was recently shown that this distillation rate holds for all cycles \cite{vrana}, that~is,}
\omega(\Tens(C_k), \Tens(k))^{-1} &= 2 \quad\textnormal{for any $k$.}\label{distillation2}
\end{align}
See \cite{vrana2015asymptotic} and \cite{vrana} for general properties of $\omega(s, t)$.
It is an open problem in algebraic complexity theory to compute the exponent of matrix multiplication $\omega = \omega(\langle 2,2,2 \rangle) = \omega(\Tens(C_3))$. The currently best bounds on this number are $2\leq \omega < 2.3728639$~\cite{le2014powers}. Rather than improving the bounds on $\omega$, we will in this paper focus on bounding $\omega(\Tens(C_k))$ for $k>3$. 

We will use the following characterizations of the exponent and subexponent, which are straightforward generalizations of results by Strassen \cite{MR929980}.

\begin{lemma}\label{power} Let $t$ be a tensor. Then, $\omega(t) = \lim_{N\to\infty} \frac1N \log_2 \rank(t^{\otimes N})$.
\end{lemma}

\begin{lemma}\label{bigoh} Let $G$ be a graph on $k$ vertices. Then,
\begin{align*}
\omega(\Tens(G)) &= \inf \{\beta \in \RR \mid \rank(\Tens_n(G)) = \Oh(n^\beta)\}\\
&= \inf \{\beta \in \RR \mid \Tens_n(G) \leq \Tens_{\Oh(n^\beta)}(k)\}.
\end{align*}
For any $n\in \NN$, $\omega(\Tens(G)) \leq \log_n \rank(T_n(G))$.
\end{lemma}

\begin{lemma}\label{subomega}
Let $G$ be a graph on $k$ vertices. Then,
\[
\omega(\Tens(G), \Tens(k))^{-1} = \sup\{\beta \in \RR\mid \Tens_n(G) \geq \Tens_{\Omega(n^\beta)}(k)\}.
\]
\end{lemma}

For any $n_1,n_2,n_3\in \NN$, the matrix multiplication tensor $\langle n_1,n_2,n_3 \rangle$ is defined as
\begin{multline*}
\langle n_1,n_2,n_3 \rangle \coloneqq \sum_{\mathclap{\substack{i \in [n_1]\times [n_2]\times [n_3]}}} (b_{i_1}\tightotimes b_{i_2})\otimes (b_{i_2}\tightotimes b_{i_3}) \otimes (b_{i_3} \tightotimes b_{i_1})\\[-0.5em]
\in (\CC^{n_1}\tightotimes \CC^{n_2}) \otimes (\CC^{n_2}\tightotimes \CC^{n_3}) \otimes (\CC^{n_3} \tightotimes \CC^{n_1}).
\end{multline*}
So $\langle n,n,n\rangle$ equals $\Tens_n(C_3)$, and $\langle n_1, n_2, n_3\rangle$ may be thought of as the tensor corresponding to a triangle with edges ``weighted'' by $n_1, n_2, n_3$.
For any real numbers $\gamma_1, \gamma_2, \gamma_3\geq 0$, define
\begin{equation}\label{dual}
\omega(\gamma_1,\gamma_2,\gamma_3) \coloneqq \inf \{\beta \in \RR \mid \rank(\langle \lfloor n^{\gamma_1} \rfloor,\lfloor n^{\gamma_2} \rfloor,\lfloor n^{\gamma_3} \rfloor\rangle) = \Oh(n^\beta)\}.
\end{equation}
By Lemma\nobreakspace \ref {bigoh}, the definition of $\omega(\gamma_1,\gamma_2,\gamma_3)$ in \eqref{dual} agrees with the definition of~$\omega$ in \eqref{omega} in the sense that $\omega(\langle 2,2,2 \rangle) = \omega(1,1,1)$.
Define the \defin{dual exponent of matrix multiplication} $\alpha$ by
\begin{equation}\label{dualexponent}
\alpha \coloneqq \sup \{ \gamma\in \RR \mid \omega(1,1,\gamma) = 2\}.
\end{equation}
The currently best bounds on this number are $0.3029805 < \alpha \leq 1$ \cite{le2012faster}. The number $\omega$ equals 2 if and only if $\alpha$ equals~1. The dual exponent $\alpha$ will turn out to be useful in combination with tensor surgery. 

We will use the following straightforward property of $\omega(\gamma_1, \gamma_2, \gamma_3)$.

\begin{lemma}\label{scaling} Let $\gamma_1, \gamma_2, \gamma_3, \delta \geq 0$ be real numbers. Then
\[
\omega(\delta\gamma_1, \delta\gamma_2, \delta\gamma_3) = \delta \exponent(\gamma_1, \gamma_2, \gamma_3).
\]
\end{lemma}
\begin{proof}
Suppose $\omega(\gamma_1, \gamma_2, \gamma_3) < \beta$. Then by definition 
\[
\rank(\langle \floor{n^{\gamma_1}}, \floor{n^{\gamma_2}}, \floor{n^{\gamma_3}}\rangle) = \Oh(n^\beta)
\]
and thus
\[
\rank(\langle \floor{n^{\delta\gamma_1}},\floor{n^{\delta\gamma_2}},\floor{n^{\delta\gamma_3}} \rangle) \leq \rank(\langle \floor{N^{\gamma_1}},\floor{N^{\gamma_2}},\floor{N^{\gamma_3}} \rangle) = \Oh(n^{\delta \beta})
\]
with $n^\delta \leq N \leq n^\delta + 1$ an integer. So $\omega(\delta\gamma_1, \delta\gamma_2, \delta\gamma_3) < \delta \beta$. Conversely, suppose that $\omega(\delta\gamma_1, \delta\gamma_2, \delta\gamma_3) < \beta$. Then by definition
\[
\rank(\langle \floor{n^{\delta\gamma_1}},\floor{n^{\delta\gamma_2}},\floor{n^{\delta\gamma_3}} \rangle) = \Oh(n^{\beta})
\]
and thus
\[
\rank(\langle \floor{N^{\gamma_1}}, \floor{N^{\gamma_2}}, \floor{N^{\gamma_3}} \rangle) \leq \rank(\langle \floor{n^{\delta\gamma_1}},\floor{n^{\delta\gamma_2}},\floor{n^{\delta\gamma_3}} \rangle) = \Oh(N^{\beta/\delta}),
\]
where $n$ is the smallest integer such that $N\leq n^\delta$. So $\omega(\gamma_1, \gamma_2, \gamma_3) < \beta/\delta$.
\end{proof}

\subsection{Lower bound methods}
Let $t\in \CC^{n_1} \otimes \cdots \otimes \CC^{n_k}$.
A flattening of $t$ is a grouping of the tensor legs into two groups as to obtain a matrix $A_t$. The flattening of a simple tensor is a simple matrix (a rank-1 matrix). Therefore, the rank of the flattening matrix $A_t$ is a lower bound for the (border) rank of the tensor $t$,
\[
\rank(A_t) \leq \borderrank(t) \leq \rank(t).
\]
Recall that for matrices, rank is multiplicative under taking the tensor product. Therefore, by  Lemma\nobreakspace \ref {power},
\[
\log_2 \rank(A_t) = \omega(A_t) \leq \omega(t).
\]

Let $G = (V,E)$ be a graph. A cut of $G$ is a partition of $V$ into two disjoint sets. The size of a cut is the number of edges crossing the cut. A maximum cut is a cut of maximum size. Let $f(G)$ denote the size of a maximum cut of~$G$. 
Let $V = V_1 \sqcup V_2$ be a cut of $G$ of maximum size $f(G)$.
Flattening the tensor $\Tens_n(G)$ along the cut yields the matrix
\[
A  = \sum_{i\in [n]^E} \Bigl(\bigotimes_{u \in V_1} \Bigl(\bigotimes_{\substack{e\in E:\\ u\in e}} b_{i_e}\Bigr)\Bigr) \otimes \Bigl(\bigotimes_{v \in V_2} \Bigl(\bigotimes_{\substack{e\in E:\\ v\in e}} b_{i_e}\Bigr)\Bigr).
\]
The rank of $A$ equals $n^{f(G)}$. Therefore
\begin{equation}\label{flatlb}
n^{f(G)} = \rank(A) \leq \borderrank(\Tens_n(G)) \leq \rank(\Tens_n(G)),
\end{equation}
In \eqref{flatlb}, taking $n=2$ and taking the logarithm $\log_2$, yields the following inequalities of graph parameters,
\begin{equation}\label{graphparam}
f(G) \leq \omega(\Tens(G)) \leq \log_2 \rank(\Tens(G)) \leq |E(G)|.
\end{equation}
For bipartite graphs, each inequality in \eqref{graphparam} is an equality.
For odd cycles we get the flattening lower bounds $n^{k-1}\leq \rank(\Tens_n(C_k))$ and $k-1\leq \omega(\Tens_2(C_k))$. There exist more sophisticated flattenings called Young flattenings \cite{landsberg2011new}, which in our language correspond to a sophisticated splitting of a vertex before flattening. Young flattenings were used in \cite{buhrman2016nondeterministic} to show that the flattening lower bound on the border rank of $\Tens_n(C_k)$ is not tight for odd $k$. 
However, we do not know of a Young flattening that improves the asymptotic maximum cut lower bound $f(G) \leq \omega(\Tens(G))$.

\section{Tensor surgery on cycles}\label{cycles}

In this section we will prove upper bounds on the tensor rank and exponent of cycle tensors using tensor surgery.

\subsection{Tensor rank}
Let $t = t^1\otimes \cdots \otimes t^k$ be a simple $k$-tensor in $\bigotimes_{j=1}^k (\CC^{a_j}\tightotimes \CC^{b_j})$. Then, for any $j$, we define the local rank $\rank_{\smash{\raisebox{-2pt}{$\scriptstyle\CC^{a_j}\otimes \CC^{b_j}$}}}(t^j)$ of $t^j$ to be the rank of $t^j$ as an element of $\CC^{a_j} \otimes \CC^{b_j}$.


\begin{theorem}
\label{explicit}
For any odd number $k\geq3$, the tensor rank of the tensor corresponding to the cycle graph $C_k$ is upper bounded by 
\[
\rank(\Tens(C_k)) \leq 2^k - 1.
\]
Moreover, $\Tens(C_k)$ has a decomposition that consists of $2^{k}-2$ simple summands whose first tensor leg has local rank $1$ and one simple summand whose first tensor leg has local rank $2$.
\end{theorem}

\begin{proof}
We prove the statement by induction on odd $k\geq 3$.
If $k = 3$, then with notation as in the introduction, Strassen's decompositions is
\begin{alignat*}{1}
\Tens(C_3)\,\, =\quad &-\,\, b_{\textsf{--}0}\otimes b_{0\textsf{+}} \otimes b_{11}\quad -\,\, b_{11}\otimes b_{\textsf{--}0} \otimes b_{0\textsf{+}}\quad -\,\, b_{0\textsf{+}}\otimes b_{11}\otimes b_{\textsf{--}0}\\[0.3em]
&+\,\, b_{\textsf{--}1}\otimes b_{1\textsf{+}}\otimes b_{00}\quad +\,\, b_{00}\otimes b_{\textsf{--}1}\otimes b_{1\textsf{+}}\quad +\,\, b_{1\textsf{+}}\otimes b_{00}\otimes b_{\textsf{--}1}\\[0.3em]
&+\,\,(b_{00}+b_{11})\otimes (b_{00}+b_{11})\otimes (b_{00}+b_{11}),
\end{alignat*}
so the statement of the theorem holds.
%
%
Assume that the statement holds for~$k = \ell$. This means that $\Tens(C_\ell) = \sum_{i=1}^{2^{\ell}-1} t^1_i \otimes \cdots \otimes t^\ell_i$ for some $t^j_i \in \CC^2 \otimes \CC^2$ such that 
\[
\#\{i \mid \rank_{\CC^2\otimes \CC^2}(t^1_i) = 1\} = 2^\ell - 2 \quad\textnormal{ and }\quad \#\{i \mid \rank_{\CC^2\otimes \CC^2}(t^1_i) = 2\} = 1.
\]
Define the linear map $\phi$ by
\begin{alignat*}{2}
\phi{}:{} &\CC^{2}\otimes \CC^{2} \to (\CC^2 \otimes \CC^2)^{\otimes 3} \\
&u \otimes v\mapsto \hspace{-1ex}\sum_{j\in\{0,1\}^2}\hspace{-1ex} (u \tightotimes b_{j_1}) \otimes (b_{j_1} \tightotimes b_{j_2}) \otimes (b_{j_2}\tightotimes v).
\end{alignat*}
Let $\psi_\ell : (\CC^2\tightotimes \CC^2)^{\otimes \ell} \to  (\CC^2\tightotimes \CC^2)^{\otimes \ell+2}$ be the map that applies $\phi$ at the first tensor leg. Then 
\[
\Tens(C_{\ell+2}) = \psi_\ell(\Tens(C_\ell)) = \sum_{i=1}^{\smash{2^\ell-1}} \phi(t^1_i) \otimes t^2_i \otimes \cdots \otimes t^\ell_i.
\]
If $\rank_{\CC^2\otimes \CC^2}(t^1_i) = 1$, then $\phi(t^1_i)$ has a decomposition of size  4 such that for every simple summand the first tensor leg has local rank 1. If $\rank_{\CC^2\otimes \CC^2}(t^1_i) = 2$, then $\phi(t^1_i)\cong\Tens(C_3)$ has a decomposition of size~7 such that for six simple summands the first tensor leg has local rank 1, while for one simple summand the first tensor leg has local rank 2. We conclude that $\Tens(C_{\ell+2})$ has rank at most $(2^{\ell}-2)4 + 1\cdot 7 = 2^{\ell+2}-1$. Moreover, $\Tens(C_{\ell+2})$ has a decomposition  that consists of $2^{\ell+2}-2$ simple summands whose first tensor leg has local rank 1 and one simple summand whose first tensor leg has local rank~2. We conclude that the statement of the theorem holds for $k = \ell+2$. 
%
\end{proof}

\begin{remark}\label{openquestion}
Before moving on, let us say something about lower bounds on the tensor rank $\rank(\Tens(C_k))$.
For any $k$, instead of flattening $\Tens(C_k)$ to a matrix, we can flatten $\Tens(C_k)$ to the 3-tensor $\langle 2,2,2^{k-2}\rangle$. As mentioned in \cite{buhrman2016nondeterministic}, by taking a Young flattening of the latter, we get the lower bound
\[
2^k-2^{k-2} + 1 \leq \borderrank(\Tens(C_k)).
\]
Applying the rank lower bound $\rank(\langle n,n,m\rangle) \geq 2mn +2n -m -2$ for $m\geq n\geq 3$ of \cite{DBLP:journals/jc/Blaser03} to $\langle 2,2,2^{k-2}\rangle$ gives
\begin{equation}\label{bl}
2^k-2^{k-2} + 2 \leq \rank(\Tens(C_k)).
\end{equation}
%
For the triangle graph, Strassen already showed that $\rank(\Tens(C_3)) \leq 7$ \cite{strassen1969gaussian} and  Winograd showed that $\rank(\Tens(C_3)) \geq 7$ \cite{winograd1971multiplication}. Only quite recently Landsberg showed that also $\borderrank(\Tens(C_3)) \geq 7$ \cite{landsberg2006border}.
For the next smallest interesting case $\Tens(C_5)$, Theorem\nobreakspace \ref {explicit} brings the rank and border rank in the following ranges:
\begin{align*}
24 &\leq \borderrank(\Tens(C_5)) \leq 31,\\
25 &\leq \rank(\Tens(C_5)) \leq 31.
\end{align*}
\end{remark}

\begin{remark}
We mention that the decomposition of $\Tens(C_5)$ given by the proof of Theorem\nobreakspace \ref {explicit} is different from the decomposition given in \cite{buhrman2016nondeterministic} in the sense of De~Groote's work~\cite{de1978varieties}, that is, the decompositions can not be transformed into each other by sandwiching and cyclic permutation. This is because the local ranks of the summands are incompatible.
\end{remark}

In general, with tensor surgery we can transform decompositions of matrix multiplication tensors $\langle n_1, n_2, n_3 \rangle$ to decompositions of $\Tens_n(C_k)$. We will illustrate this with $\Tens_4(C_5)$. We make use of the bounds $\rank(\langle 4,4,2\rangle) \leq 26$, $\rank(\langle 4, 4, 4\rangle) \leq 49$ and $\borderrank(\langle 4,4,2\rangle) \leq 24$, $\borderrank(\langle 4,4,4 \rangle) \leq 46$, see \cite{hopcroft1971minimizing,smirnov2013bilinear}. First note that $\rank(\Tens_4(C_5)) \leq \rank(\Tens_2(C_5))^2 \leq 31^2$.

\begin{proposition}\label{further}
$\rank(\Tens_4(C_5)) \leq 937 < 31^2$ \!and\, $\borderrank(\Tens_4(C_5)) \leq 910$.
\end{proposition}
\begin{proof}
Let $\phi$ be the linear map $\CC^4\tightotimes \CC^4 \mapsto (\CC^4\tightotimes \CC^4)^{\otimes 3}$ defined on simple tensors by $u \otimes v \mapsto \sum_{j\in[4]^2} (u\tightotimes b_{j_1})\otimes (b_{j_1} \tightotimes b_{j_2})\otimes(b_{j_2}\tightotimes v)$, and let $\psi$ be the linear map $(\CC^4\tightotimes \CC^4)^{\otimes 3} \to (\CC^4\tightotimes \CC^4)^{\otimes 5}$ which applies $\phi$ to the first tensor leg. Then $\Tens_4(C_5)$ equals $\psi(\langle 4,4,4 \rangle)$. Taking the tensor square of Strassen's decomposition gives a decomposition $\sum_{i=1}^{49} t^1_i \otimes \cdots \otimes t^k_i$ of $\langle 4,4,4\rangle$ such that
\begin{align*}
\#\{i \mid \rank_{\CC^2\otimes \CC^2}(t^1_i) = 1\} &= 6^2, \\
\#\{i \mid \rank_{\CC^2\otimes \CC^2}(t^1_i) = 2\} &= 6+6, \\
\#\{i \mid \rank_{\CC^2\otimes \CC^2}(t^1_i) = 4\} &= 1 .
\end{align*}
If $\rank_{\CC^2\otimes \CC^2}(t^1_i) = 1$, then $\phi(t^1_i)$ has rank $4^2$. If $\rank_{\CC^2\otimes \CC^2}(t^1_i) = 2$, then $\phi(t^1_i) \cong \langle 4,4,2 \rangle$ has rank at most 26. If $\rank_{\CC^2\otimes \CC^2}(t^1_i) = 4$, then $\phi(t^1_i) \cong \langle 4,4,4 \rangle$ has rank at most 49.
Therefore, applying $\psi$ to the simple summands $t^1_i\otimes \cdots \otimes t^k_i$ we obtain $\rank(\Tens_4(C_5)) \leq 6^2 \cdot 4^2 + 12 \cdot 26 + 1 \cdot 49 = 937$.

For the border rank, we have $\borderrank(\langle 4,4,2\rangle) \leq 24$ and $\borderrank(\langle 4,4,4 \rangle) \leq 46$, so that by the same argument $\borderrank(\Tens_4(C_5)) \leq 6^2 \cdot 4^2 + 12 \cdot 24 + 1 \cdot 46 = 910$.
\end{proof}

\subsection{Exponent}
In view of the lower bound \eqref{bl}, our tensor rank bounds might not look that strong. We will now see, however, that applying the same techniques in the asymptotic setting yields  optimal bounds, assuming $\omega=2$.
%
Let $\omega_k \coloneqq \omega(\Tens(C_k))$. 





\begin{theorem}\label{easy}
For $k,\ell$ odd, \,$\omega_{k+\ell-1} \leq \omega_k + \omega_\ell$.
\end{theorem}

The idea of the proof is to take the $k$-cycle $C_k$, split one vertex in $C_k$ into two vertices and insert $\ell-2$ new vertices in the graph together with the appropriate $\ell-1$ edges in order to create the $(k+\ell-1)$-cycle. In pictures, for $k=5$ and $\ell=3$,
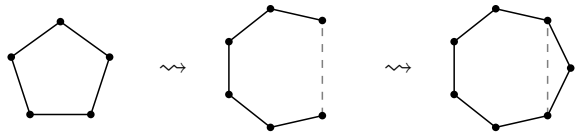
\[
\begin{minipage}{1.5cm}
\begin{tikzpicture}[vertex/.style = {circle, fill, black, minimum width = 1.mm, inner sep=0pt}]
    \path[coordinate] (0,0)  coordinate(A)
                ++( 2*1*36:0.8cm) coordinate(B)
                ++( 2*2*36:0.8cm) coordinate(C)
                ++( 2*3*36:0.8cm) coordinate(D)
                ++( 2*4*36:0.8cm) coordinate(E);
	\draw [line width=0.2mm] (D) node [vertex] {} -- (E) node [vertex] {} -- (A) node [vertex] {} -- (B) node [vertex] {} -- (C) node [vertex] {} -- (D);
\end{tikzpicture}
\end{minipage}
\quad\leadsto\quad
\begin{minipage}{1.6cm}
\begin{tikzpicture}[vertex/.style = {circle, fill, black, minimum width = 1.mm, inner sep=0pt}]
    \path[coordinate] (0,0)  coordinate(A)
                ++( 2*1*180/7+12.857142857:0.7cm) coordinate(B1)
                ++( 2*2*180/7+12.857142857:0.7cm) coordinate(B)
                ++( 2*3*180/7+12.857142857:0.7cm) coordinate(C)
                ++( 2*4*180/7+12.857142857:0.7cm) coordinate(D)
                ++( 2*5*180/7+12.857142857:0.7cm) coordinate(E)
                ++( 2*6*180/7+12.857142857:0.7cm) coordinate(F);
	\draw [line width=0.2mm, dashed, gray] (A) -- (B);
	\draw [line width=0.2mm] (D) node [vertex] {} -- (E) node [vertex] {} -- (F) node [vertex] {} -- (A) node [vertex] {} (B) node [vertex] {} -- (C) node [vertex] {} -- (D);

\end{tikzpicture}
\end{minipage}
\quad\leadsto\quad
\begin{minipage}{1.6cm}
\begin{tikzpicture}[vertex/.style = {circle, fill, black, minimum width = 1.mm, inner sep=0pt}]
    \path[coordinate] (0,0)  coordinate(A)
                ++( 2*1*180/7+12.857142857:0.7cm) coordinate(B1)
                ++( 2*2*180/7+12.857142857:0.7cm) coordinate(B)
                ++( 2*3*180/7+12.857142857:0.7cm) coordinate(C)
                ++( 2*4*180/7+12.857142857:0.7cm) coordinate(D)
                ++( 2*5*180/7+12.857142857:0.7cm) coordinate(E)
                ++( 2*6*180/7+12.857142857:0.7cm) coordinate(F);
	\draw [line width=0.2mm, dashed, gray] (A) -- (B);
	\draw [line width=0.2mm] (D) node [vertex] {} -- (E) node [vertex] {} -- (F) node [vertex] {} -- (A) node [vertex] {} -- (B1) node [vertex] {}  --  (B) node [vertex] {} -- (C) node [vertex] {} -- (D);
\end{tikzpicture}
\end{minipage}
\]
Next we consider an optimal decomposition of $\Tens_n(C_k)$. Not only inserting $\ell-1$ edges comes with a cost, but also splitting the vertex. The crucial observation is that the total cost is at most the cost of creating an $\ell$-cycle, which is asymptotically $\omega_\ell$.

\begin{proof}
Let $\phi$ be the linear map $\CC^{n}\tightotimes \CC^n \to (\CC^{n}\tightotimes \CC^n)^{\otimes \ell}$ defined on simple tensors by $u\otimes v \mapsto \sum_{j\in[n]^{\ell-1}} (u\tightotimes b_{j_1}) \otimes (b_{j_1} \tightotimes b_{j_2}) \otimes \cdots \otimes (b_{j_{\ell-1}}\tightotimes v)$, and let $\psi$ be the linear map $(\CC^n\tightotimes\CC^n)^{\otimes k} \to (\CC^n\tightotimes\CC^n)^{\otimes k+\ell-1}$ that applies~$\phi$ at the first tensor leg. Then $\Tens_n(C_{k+\ell-1}) = \psi(\Tens_n(C_k))$. Let $\varepsilon > 0$. Then there is a constant $c_\varepsilon\in\NN$ and a decomposition of $\Tens_n(C_k)$ as a sum of at most~$c_\varepsilon n^{\omega_k + \varepsilon}$ simple summands (Lemma\nobreakspace \ref {bigoh}). Consider one simple summand $t^1\otimes \cdots \otimes t^k$ in this decomposition. We have $\rank_{\CC^n\otimes \CC^n}(t^1)\leq n$ and hence $\phi(t^1) \leq \Tens_n(C_\ell)$. The rank of $\psi(t^1\otimes \cdots \otimes t^k)$ is therefore at most $d_\varepsilon n^{\omega_\ell + \varepsilon}$ for some constant $d_\varepsilon\in \NN$. We conclude that the rank of $\psi(\Tens_n(C_k))$ is at most $c_\varepsilon d_\varepsilon n^{\omega_k + \omega_\ell + 2\varepsilon}$, and thus $\omega_{k+\ell-1} \leq \omega_k + \omega_\ell$ (Lemma\nobreakspace \ref {bigoh}).
\end{proof}

\begin{corollary}
Let $k\geq 5$ odd. Then, $\omega_k \leq \omega_{k-2} + \omega_3$ and thus $\omega_k \leq \frac{k-1}{2} \omega$.
\end{corollary}


\begin{corollary}
If $\omega = 2$, then $\omega_k = k-1$ for all odd $k$.
\end{corollary}


\begin{remark}\label{localrank}
The proofs of Theorem\nobreakspace \ref {explicit} and Proposition\nobreakspace \ref {further} crucially relied on a careful local rank analysis of Strassen's decomposition and other decompositions of matrix multiplication tensors. The same technique may be applied in the asymptotic setting to improve the results of Theorem\nobreakspace \ref {easy}, in the following sense.  Suppose one has a specific upper bound for $\omega_k$ together with information about the local ranks in the corresponding decomposition of $\Tens_n(C_k)$ for any~$n$.  Then, when applying the surgery map~$\psi$ to such a decomposition, as in the proof of Theorem\nobreakspace \ref {easy}, one can use the specific local rank information instead of using the worst-case upper bound $\rank_{\CC^n \otimes \CC^n}(t^1)\leq n$, and thus obtain an improved asymptotic bound. 

The local rank viewpoint reveals an interesting fact about the decompositions of cycle tensors, which is also relevant for the asymptotic local rank analysis idea. Namely, take the tensor $\Tens_n(C_k)$ and let $\psi$ be the map that split one of the vertices,
\[
\psi : (\CC^n\otimes \CC^n)^{\otimes k} \to (\CC^n\otimes \CC^n)^{\otimes k-1} \otimes \CC^n \otimes \CC^n.
\]
Then $\psi(\Tens_n(C_k))=\Tens_n(L_k)$ where $L_k$ is the linear graph with $k$ edges, and hence we have  $\rank(\psi(\Tens_n(C_k))) = n^k$. Therefore, if $\Tens_n(C_k) = \sum_{i=1}^r t^1_i \otimes \cdots \otimes t^k_i$ is a decomposition into simple tensors, then for any $j\in [k]$ we have
\[
\sum_{i=1}^r \rank_{\CC^n\otimes \CC^n}(t^j_i) \geq n^k.
\]
Let $r = n^{\beta}$ and let $\Tens_n(C_k)=\sum_{i=1}^r t^1_i \otimes \cdots \otimes t^k_i$ be a decomposition. Then the average local rank at the $j$th leg is lower bounded by
\[
\frac{1}{r}\sum_{i=1}^r \rank_{\CC^n\otimes \CC^n}(t^j_i) \geq n^{k-\beta},
\]
while $\max_{i\in [r]} \rank_{\CC^n\otimes \CC^n}(t^j_i) \leq n$. If $\beta$ is close to  $k-1$, then the average local rank is close to the maximum. However, if $\beta$ is bounded away from~$k-1$ then there is a gap between average and maximum local rank, so that an improvement by local rank analysis as described above is possible.
\end{remark}

The following theorem gives an upper bound on $\omega_k$ in terms of the dual exponent of matrix multiplication~$\alpha$.

\begin{theorem}\label{asymp}
For any odd $k\geq 3$, $\omega_k \leq k - \alpha\bigl(1+ \frac{1-\alpha}{k-1+\alpha}\bigr) \leq k - \alpha$. 
\end{theorem}

The idea of the proof is as follows. Start with the unbalanced triangle tensor $\langle n,n,\lfloor n^\alpha \rfloor \rangle$. On the graph level, we split a vertex, and insert a vertex with two edges:
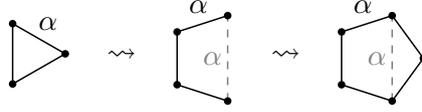
\[
\begin{minipage}{0.8cm}
\begin{tikzpicture}[vertex/.style = {circle, fill, black, minimum width = 1.mm, inner sep=0pt}]
    \path[coordinate] (0,0)  coordinate(A)
                ++( 2*1*60+30:0.8cm) coordinate(B)
                ++( 2*2*60+30:0.8cm) coordinate(C);

	\draw [line width=0.2mm]  (A) node [vertex] {} -- node[xshift=1.2mm, yshift=1.2mm]{\smash{$\alpha$}} (B) node [vertex] {} -- (C) node [vertex] {} -- (A);
\end{tikzpicture}
\end{minipage}
\quad\leadsto\quad
\begin{minipage}{0.8cm}
\begin{tikzpicture}[vertex/.style = {circle, fill, black, minimum width = 1.mm, inner sep=0pt}]
    \path[coordinate] (0,0)  coordinate(A)
                ++( 2*1*180/5-18:0.7cm) coordinate(B1)
                ++( 2*2*180/5-18:0.7cm) coordinate(B)
                ++( 2*3*180/5-18:0.7cm) coordinate(C)
                ++( 2*4*180/5-18:0.7cm) coordinate(D);
	\draw [line width=0.2mm, dashed, gray] (A) -- node[xshift=-2mm, yshift=0mm]{$\alpha$} (B);

	\draw [line width=0.2mm] (D) node [vertex] {}  -- (A) node [vertex] {}  (B) node [vertex] {} -- node[xshift=-0.5mm, yshift=1.5mm]{\smash{$\alpha$}} (C) node [vertex] {} -- (D);

\end{tikzpicture}
\end{minipage}
\quad\leadsto\quad
\begin{minipage}{0.8cm}
\begin{tikzpicture}[vertex/.style = {circle, fill, black, minimum width = 1.mm, inner sep=0pt}]
    \path[coordinate] (0,0)  coordinate(A)
                ++( 2*1*180/5-18:0.7cm) coordinate(B1)
                ++( 2*2*180/5-18:0.7cm) coordinate(B)
                ++( 2*3*180/5-18:0.7cm) coordinate(C)
                ++( 2*4*180/5-18:0.7cm) coordinate(D);
	\draw [line width=0.2mm, dashed, gray] (A) -- node[xshift=-2mm, yshift=0mm]{$\alpha$} (B);
	\draw [line width=0.2mm] (D) node [vertex] {}  -- (A) node [vertex] {} -- (B1) node [vertex] {}  -- (B) node [vertex] {} -- node[xshift=-0.5mm, yshift=1.5mm]{\smash{$\alpha$}} (C) node [vertex] {} -- (D);
\end{tikzpicture}
\end{minipage}
\]
The crucial observation is that the total cost of splitting a vertex and inserting one vertex with the two appropriate edges is $\omega(\langle n, n, \floor{n^{\alpha}}\rangle)$ which is 2. Repeating this procedure $(k-1)/2$ times yields $\Tens_n(C_k)$ but with edges ``weighted'' by $\floor{n^\alpha},n, \ldots, n$ respectively, at cost $k-1$ in the exponent. To get an evenly weighted $\Tens_n(C_k)$ we symmetrise cyclically.

\begin{proof}
Let $0<\gamma<\alpha$. Let $\Tens_{n,\gamma}(C_\ell)$ be the cycle tensor with edges weighted by $\floor{n^\gamma},n,\ldots, n$ respectively,
\[
\Tens_{n,\gamma}(C_\ell) = \sum_{\mathclap{i\in [\floor{n^\gamma}] \times[n]^{\times(\ell-1)}}} (b_{i_1} \tightotimes b_{i_2}) \otimes (b_{i_2} \tightotimes b_{i_3}) \otimes \cdots \otimes (b_{i_\ell} \tightotimes b_{i_1}).
\]
We will show that $\rank(\Tens_{n,\gamma}(C_k)) = \Oh(n^{k-1+\varepsilon})$ for all $\varepsilon>0$ by induction on odd $k\geq 3$. For $k=3$, the statement is true by definition of $\alpha$. Suppose the statement holds for $k=\ell$. Let $\phi$ be the linear map 
$\mathbb C^{\lfloor n^{\gamma} \rfloor} \otimes \mathbb C^n \to (\mathbb
C^{\lfloor n^{\gamma} \rfloor} \otimes \mathbb C^n) \otimes (\mathbb C^n
\otimes \mathbb C^n)^{\otimes 2}$
defined on simple tensors by 
\[
u\otimes v \mapsto \sum_{j\in[n]^2} (u\tightotimes b_{j_1}) \otimes (b_{j_1} \tightotimes b_{j_2}) \otimes (b_{j_2}\tightotimes v),
\]
and let $\psi_\ell$ be the linear map 
$(\mathbb C^{\lfloor n^{\gamma} \rfloor} \otimes \mathbb C^n) \otimes (\mathbb
C^n \otimes \mathbb C^n)^{\otimes \ell - 2} \otimes (\mathbb C^n \otimes
\mathbb C^{\lfloor n^{\gamma} \rfloor}) \to (\mathbb C^{\lfloor n^{\gamma}
\rfloor} \otimes \mathbb C^n) \otimes (\mathbb C^n \otimes \mathbb
C^n)^{\otimes \ell} \otimes (\mathbb C^n \otimes \mathbb C^{\lfloor n^{\gamma}
\rfloor})$
that applies~$\phi$ at the first tensor leg. 
Then,
%
\[
\Tens_{n,\gamma}(C_{\ell+2}) = \psi_\ell(\Tens_{n,\gamma}(C_\ell)).
\]
Let $\varepsilon > 0$. There is a constant $c_\varepsilon\in \NN$ and a decomposition of $\Tens_{n,\gamma}(C_\ell)$ as a sum of at most $c_\varepsilon n^{(\ell-1) + \varepsilon}$ simple summands (Lemma\nobreakspace \ref {bigoh}).
Consider one simple summand $t^1\otimes \cdots \otimes t^\ell$ in this decomposition. We have $\rank_{\CC^{\floor{n^\gamma}}\otimes \CC^n}(t^1)\leq \floor{n^\gamma}$ and hence $\phi(t^1) \leq \Tens_{n,\gamma}(C_3)$. The rank of $\psi(t^1\otimes \cdots \otimes t^k)$ is therefore at most $d_\varepsilon n^{2 + \varepsilon}$ for some constant $d_\varepsilon\in \NN$. We conclude that the rank of $\psi(\Tens_{n,\gamma}(C_{\ell+2}))$ is at most $c_\varepsilon d_\varepsilon n^{\ell-1 + 2 + 2\varepsilon} = c_\varepsilon d_\varepsilon n^{\ell + 1 + 2\varepsilon}$, and thus $\rank(\Tens_{n,\gamma}(C_{\ell+2})) = \Oh(n^{\ell+1+\varepsilon})$ for any $\varepsilon>0$.


Symmetrizing $\Tens_{n,\gamma}(C_k)$ cyclically gives us a balanced cycle tensor, as follows: 
\[
\Tens_{n^{k-1}\floor{n^{\gamma}}}(C_k) \cong \bigotimes_{\pi} \pi  \cdot\Tens_{n,\gamma}(C_k),
\]
where $\pi$ goes over all powers of the cyclic permutation $(12\cdots k)$, and $\pi$ acts by permuting the tensor legs.
Let $\varepsilon > 0$ and let $\gamma < \alpha$. Then, by submultiplicativity of tensor rank,
\[
\rank(\Tens_{n^{k-1}\floor{n^{\gamma}}}(C_k)) \leq \rank(\Tens_{n,\gamma}(C_k))^k \leq c_{\varepsilon,\gamma}^k n^{(k-1+\varepsilon)k}.
\]
Then, by Lemma\nobreakspace \ref {bigoh},
\[
\omega(\Tens(C_k)) \leq \frac{\log_n(c_{\varepsilon, \gamma}^k n^{(k-1+\varepsilon)k})}{\log_n(n^{k-1}\floor{n^\gamma})} \leq \frac{\log_n c_{\varepsilon, \gamma}^k + (k-1+\varepsilon)k}{k-1+\gamma-o(1)}.
\]
Letting $n\to\infty$, $\varepsilon\to0$, $\gamma \to \alpha$ gives
\[
\omega(\Tens(C_k)) \leq \frac{k}{k-1+\alpha} (k-1) = k-\alpha \Bigl(1 + \frac{1-\alpha}{k-1+\alpha}\Bigr),
\]
finishing the proof.
%
\end{proof}

\begin{remark}
We can naturally define $\omega(\gamma_1, \ldots, \gamma_k)$ by extending the definition in \eqref{dual}. An interesting intermediate result in the above proof of Theorem\nobreakspace \ref {asymp} is that for any $k\geq3$ and any $0 < \gamma < \alpha$ we have $\omega(1,1, \ldots, 1, \gamma) = k-1$. The standard flattening argument implies that this bound is optimal. Also, the observation in Remark\nobreakspace \ref {localrank} applied to $\Tens_{n,\gamma}(C_k)$ implies that the decompositions achieving the exponent $k-1+\varepsilon$ must have close to maximal local ranks, and thus the surgery method cannot be improved by taking into account local rank information.
\end{remark}

Summarizing, the following table contains the best bounds on the exponent of odd cycles $\omega_k = \omega(\Tens(C_k))$ for some small odd $k$. From $k=11$ onwards, Theorem\nobreakspace \ref {asymp} gives the best upper bound. This bound converges to $k-\alpha$ when we let $k\to\infty$.

\begin{center}
\begin{tabular}{llll}
\toprule
$k$ & \multicolumn{2}{c}{$\omega_k$} & reference \\\cmidrule{2-3}
	& lower & upper &\\
\midrule
3  & 2 & 2.3728639 & \cite{le2014powers} \\
5  & 4 & 4.6031719 & \cite{buhrman2016nondeterministic} \\
7  & 6 & 6.6511249 & \cite{buhrman2016nondeterministic} \\
9  & 8 & 8.6715848 & Theorem\nobreakspace \ref {asymp} \\
11 & 10 & 10.676522 & Theorem\nobreakspace \ref {asymp} \\
13 & 12 & 12.679854 & Theorem\nobreakspace \ref {asymp} \\
\bottomrule
\end{tabular}
\end{center}


\subsection{Covering and distilling} In some cases, we have another method for obtaining upper bounds on $\omega(\Tens(G))$ which gives weaker results than the tensor surgery upper bounds above but which is conceptually easier.
%
The idea is to cover the graph $G$ with triangles, which cost $\omega$ each and use distillation to remove unwanted edges.

For example, for $k=5$, the distillation result \eqref{distillation2} says that asymptotically one copy of $\Tens(C_5)$ can be restricted to the tensor product of two copies of $\Tens(5) = \sum_{i\in\{0,1\}}b_i^{\otimes 5}$. Covering the complete graph $K_5$ with 10 triangles, gives, with subscripts denoting tensor leg positions,
\[
\Tens_{n^3}(K_5) \cong \bigotimes_{\substack{G\subseteq K_5:\\ G\cong C_3}} \Tens_n(C_3)_{V(G)} \otimes (1 \otimes \cdots \otimes 1)_{[k]\setminus V(G)},
\]
where the tensor product is over subgraphs $G$ of $K_5$ isomorphic to $C_3$. We can view $\Tens_{n^3}(K_5)$ as the tensor product of $\Tens_{n^3}(C_5)$ and a permuted copy of $\Tens_{n^3}(C_5)$.
Distilling a unit tensor $\Tens_{\Omega(n^{2\cdot 3 - \varepsilon})}(5)$ from one of these copies (Lemma\nobreakspace \ref {subomega}) gives
\[
\Tens_{n^3}(C_5)^{\oplus \Omega(n^{2\cdot 3 - \varepsilon})} \cong \Tens_{n^3}(C_5) \otimes \Tens_{\Omega(n^{2\cdot 3-\varepsilon})}(5) \leq \Tens_{n^3}(K_5) .
\]
By the asymptotic sum inequality for cycles (Proposition 27 in \cite{buhrman2016nondeterministic}) we obtain the inequality $\omega(\Tens(C_5)) \leq (10\omega-2\cdot 3)/3$ which is at most $5.90955$ by Le Gall's upper bound on $\omega$.

A variation on the above idea is to cover the cycle $C_k$ by unbalanced triangles with edge-multiplicities $(1,1,\alpha)$, which cost 2 each, and then distil a $k$-cycle with multiplicity $\alpha$. This yields $\omega(\Tens(C_k)) \leq k - \alpha$. 

\section{Tensor surgery on general graphs and hypergraphs}\label{hyper}

In this final section we want to illustrate tensor surgery on general graphs and hypergraphs. The first example shows that tensor surgery on a graph might involve absorbing a virtual hyperedge. The second example is an example of general hypergraph surgery.
We believe that the bounds in this section cannot be obtained by using only the covering and distilling technique mentioned at the end of the previous section. 

\subsection{The dome tensor}
In both examples we use the following hypergraph tensor, of which we will first establish some properties.

We define $\dome_{k,\ell}$ to be the following hypergraph on 4 vertices with multi-edges
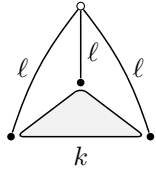
\[
\begin{tikzpicture}[vertex/.style = {circle, fill, black, minimum width = 1.mm, inner sep=0pt}, novertex/.style = {minimum width = 0.mm, inner sep=0pt}, nofillvertex/.style = {draw, circle, black, minimum width = 1.mm, inner sep=0pt}]
	\node[nofillvertex] (0,0) (X) {};
	\node[vertex, below= 9mm of X] (Y) {};
	\coordinate[novertex, below= 0.1mm of Y] (Y1);
	\node[vertex, below right= 9mm of Y, xshift=2mm] (Z) {};
	\coordinate[novertex, left= 0.1mm of Z] (Z1);
	\node[vertex, below left= 9mm of Y, xshift=-2mm] (V) {};
	\coordinate[novertex, right= 0.1mm of V] (V1);

	\draw[line width=0.2mm] (X) -- node[right, xshift=-0.5mm, yshift=-1mm] {$\small\textsf{\upshape$\ell$}$} (Y);
	\draw[line width=0.2mm] (X) edge[bend left=10] node[right] {$\small\textsf{\upshape$\ell$}$} (Z);
	\draw[line width=0.2mm] (X) edge[bend right=10] node[left] {$\small\textsf{\upshape$\ell$}$} (V);

	\draw[rounded corners=1mm, fill=gray!10, line width=0.2mm] (Y1) -- (Z1) -- node[below] {$\small\textsf{\upshape$k$}$} (V1) -- cycle;
\end{tikzpicture}
\]
where $k$ and $\ell$ denote edge-multiplicities.

\begin{lemma}\label{domelemma}
We have $3\leq \omega(\Tens(\dome_{1,1})) \leq 3\omega / 2$.
\end{lemma}
\begin{proof}
The lower bound $3\leq \omega(\Tens(\dome_{1,1}))$ is obtained by grouping the black vertices together and taking the corresponding flattening of $\Tens(\dome_{1,1})$.
For the upper bound, first observe that the exponent of the tensor $\Tens(G)$ corresponding to the graph $G$ given by
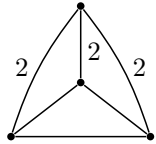
\[
\begin{tikzpicture}[vertex/.style = {circle, fill, black, minimum width = 1.mm, inner sep=0pt}, novertex/.style = {minimum width = 0.mm, inner sep=0pt}, nofillvertex/.style = {draw, circle, black, minimum width = 1.mm, inner sep=0pt}]
	\node[vertex] (0,0) (X) {};
	\node[vertex, below= 9mm of X] (Y) {};
	\coordinate[novertex, below= 0.1mm of Y] (Y1);
	\node[vertex, below right= 9mm of Y, xshift=2mm] (Z) {};
	\coordinate[novertex, left= 0.1mm of Z] (Z1);
	\node[vertex, below left= 9mm of Y, xshift=-2mm] (V) {};
	\coordinate[novertex, right= 0.1mm of V] (V1);

	\draw[line width=0.2mm] (X) -- node[right, xshift=-0.5mm, yshift=-1mm] {$\small\textsf{\upshape$2$}$} (Y);
	\draw[line width=0.2mm] (X) edge[bend left=10] node[right] {$\small\textsf{\upshape$2$}$} (Z);
	\draw[line width=0.2mm] (X) edge[bend right=10] node[left] {$\small\textsf{\upshape$2$}$} (V);
	\draw[line width=0.2mm] (Z) edge node[left] {} (V);
	\draw[line width=0.2mm] (Z) edge node[left] {} (Y);
	\draw[line width=0.2mm] (Y) edge node[left] {} (V);
	%
\end{tikzpicture}
\]
is at most $3\omega$, since $\Tens(G)$ can be obtained by combining three copies of~$\Tens(C_3)$. The distillation result \eqref{distillation} says that $\omega(\Tens(C_3), \Tens(3))^{-1} = 2$. This means that for any $\varepsilon>0$ we can restrict $\Tens_n(C_3)$ to $\Tens_{\Omega(n^{2-\varepsilon})}(T(3))$ (Lemma\nobreakspace \ref {subomega}). Applying this observation to the copy of $C_3$ that forms the base triangle in $G$ gives that
%
%
for any $\varepsilon>0$ the tensor $\Tens_{\Omega(n^{2-\varepsilon})}(\dome_{1,1})$ has rank $\Oh(n^{3\omega + \varepsilon})$. Therefore, $\omega(\Tens(\dome_{1,1})) = 3\omega$.
\end{proof}

\begin{lemma}\label{domelemma2}
$\omega(\Tens(\dome_{1,4})) = 12$.
\end{lemma}
\begin{proof}
The lower bound $3\cdot 4 \leq \omega(\Tens(\dome_{1,4}))$ is obtained by grouping the black vertices together and taking the corresponding flattening of $\Tens(\dome_{1,4})$.
We prove the upper bound by proving that $\omega(\Tens(\dome_{1,4})^{\otimes 2}) \leq 24$. We will do this by following the strategy of the proof of Lemma\nobreakspace \ref {domelemma}.
Recall that $\omega(\langle n, n^4, n^4 \rangle)$ equals $\exponent(1,4,4) = 4\exponent(\tfrac14,1,1)$ (Lemma\nobreakspace \ref {scaling}) and this number equals $4\cdot2 = 8$, since the dual exponent of matrix multiplication~$\alpha$ is at least $0.3029805$ which is strictly more than $\tfrac14$ (see \eqref{dualexponent}). Therefore, the exponent of the tensor $\Tens(G)$ corresponding to the graph $G$ given by
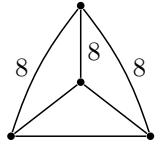
\[
\begin{tikzpicture}[vertex/.style = {circle, fill, black, minimum width = 1.mm, inner sep=0pt}, novertex/.style = {minimum width = 0.mm, inner sep=0pt}, nofillvertex/.style = {draw, circle, black, minimum width = 1.mm, inner sep=0pt}]
	\node[vertex] (0,0) (X) {};
	\node[vertex, below= 9mm of X] (Y) {};
	\coordinate[novertex, below= 0.1mm of Y] (Y1);
	\node[vertex, below right= 9mm of Y, xshift=2mm] (Z) {};
	\coordinate[novertex, left= 0.1mm of Z] (Z1);
	\node[vertex, below left= 9mm of Y, xshift=-2mm] (V) {};
	\coordinate[novertex, right= 0.1mm of V] (V1);

	\draw[line width=0.2mm] (X) -- node[right, xshift=-0.5mm, yshift=-1mm] {$\small\textsf{\upshape$8$}$} (Y);
	\draw[line width=0.2mm] (X) edge[bend left=10] node[right] {$\small\textsf{\upshape$8$}$} (Z);
	\draw[line width=0.2mm] (X) edge[bend right=10] node[left] {$\small\textsf{\upshape$8$}$} (V);
	\draw[line width=0.2mm] (Z) edge node[above,yshift=-0.5mm] {} (V);
	\draw[line width=0.2mm] (Z) edge node[right, xshift=-2mm, yshift=2mm] {} (Y);
	\draw[line width=0.2mm] (Y) edge node[left, xshift=2mm, yshift=2mm] {} (V);
	%
\end{tikzpicture}
\]
is at most $3\exponent(1,4,4) = 24$. For any $\varepsilon>0$ we can restrict $\Tens_{n}(C_3)$ to $\Tens_{\Omega(n^{2-\varepsilon})}(3)$ (Lemma\nobreakspace \ref {subomega}).
So for any $\varepsilon>0$ the tensor $\Tens_{\Omega(n^{2-\varepsilon})}(\dome_{1,4})$ has rank $\Oh(n^{24 + \varepsilon})$, which means that the inequality $\omega(\Tens(\dome_{1,4})^{\otimes 2}) \leq 24$ holds (Lemma\nobreakspace \ref {bigoh}).
%
\end{proof}

\subsection{Tensor surgery for graphs with hypergraph insertion}

The aim of the first example is to show how tensor surgery on a graph may involve the absorption of a virtual hyperedge.
Let $G$ be the multigraph 
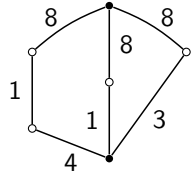
\[
\begin{tikzpicture}[vertex/.style = {circle, fill, black, minimum width = 1.mm, inner sep=0pt}, nofillvertex/.style = {draw, circle, black, minimum width = 1.mm, inner sep=0pt}, novertex/.style = {minimum width = 0.mm, inner sep=0pt}]
	\node[vertex] (0,0) (6) {};
	\node[nofillvertex, below= 9mm of 6] (4) {};
	\node[nofillvertex, left= 9mm of 4, yshift=4mm] (5) {};
	\node[nofillvertex, right= 9mm of 4, yshift=4mm] (3) {};
	\node[vertex, below= 9mm of 4] (2) {};
	\node[nofillvertex, below= 9mm of 5] (1) {};

	\draw[line width=0.2mm] (6) edge node[right]{\small\textsf{\upshape{8}}} (4);
	\draw[line width=0.2mm] (6) edge[bend left=10] node[right, yshift=1mm]{\small\textsf{\upshape{8}}} (3);
	\draw[line width=0.2mm] (6) edge[bend right=10] node[left, yshift=1mm]{\small\textsf{\upshape{8}}} (5);
	\draw[line width=0.2mm] (5) edge node[left]{\small{\textsf{\upshape{1}}}} (1);
	\draw[line width=0.2mm] (4) edge node[left]{\small{\textsf{\upshape{1}}}} (2);
	\draw[line width=0.2mm] (3) edge node[below right=-1mm]{\small\textsf{\upshape{3}}} (2);
	\draw[line width=0.2mm] (2) edge node[below]{\small\textsf{\upshape{4}}} (1);

%
%
%
%
\end{tikzpicture}
\]
where the numbers denote edge-multiplicity.
Grouping the white vertices together and grouping the black vertices together shows that the size of a max-cut is at least 32. Therefore, $\omega(\Tens(G)) \geq 32$. On the other hand, one can cover the 5-cycle on the left at cost $\omega_5$ and the remaining edges at cost~1 each, which implies that $\omega(\Tens(G)) \leq \omega(\Tens(C_5)) + 28$. Therefore, by Theorem\nobreakspace \ref {easy} if $\omega=2$, then $\omega(\Tens(G)) = 32$. We will now prove this bound independently of~$\omega$ being 2. 

\begin{proposition}\label{dome}
$\omega(\Tens(G)) = 32$.
\end{proposition}
\begin{proof}
It remains to show the upper bound. We start off with the rectangular matrix multiplication tensor $\langle n,n^4, n^4 \rangle$ at cost $\omega(1,4,4) = 4\exponent(\tfrac14,1,1) = 8$ (by Lemma\nobreakspace \ref {scaling} and since $\tfrac14< \alpha$), and, viewing it as a triangle graph
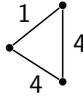
\[
\begin{minipage}{1.0cm}
\begin{tikzpicture}[vertex/.style = {circle, fill, black, minimum width = 1.mm, inner sep=0pt}, novertex/.style = {minimum width = 0.mm, inner sep=0pt}]
	\node[vertex] (0,0) (4) {};
	\node[vertex, below= 9mm of 4] (2) {};
	\node[vertex, below = 4.5mm of 4, xshift=-7mm] (1) {};

	\draw[line width=0.2mm] (4) edge node[above left=-1mm]{\small{\textsf{\upshape{1}}}} (1);
	\draw[line width=0.2mm] (4) edge node[right]{\small{\textsf{\upshape{4}}}} (2);
	\draw[line width=0.2mm] (2) edge node[below]{\small\textsf{\upshape{4}}} (1);

%
%
%
%
\end{tikzpicture}
\end{minipage}
\]
split up one of the low-dimension vertices into three vertices such that the resulting tensor corresponds to the following graph:
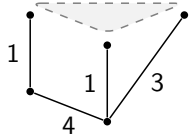
\[
\begin{minipage}{3.0cm}
\begin{tikzpicture}[vertex/.style = {circle, fill, black, minimum width = 1.mm, inner sep=0pt}, novertex/.style = {minimum width = 0.mm, inner sep=0pt}]
	\coordinate[novertex] (0,0) (6) {};
	\node[vertex, below= 9mm of 6] (4) {};
	\coordinate[above = 1mm of 4] (4a) {};
	\node[vertex, left= 9mm of 4, yshift=4mm] (5) {};
	\coordinate[above = 1mm of 5] (5a) {};
	\node[vertex, right= 9mm of 4, yshift=4mm] (3) {};
	\coordinate[above = 1mm of 3] (3a) {};
	\node[vertex, below= 9mm of 4] (2) {};
	\node[vertex, below= 9mm of 5] (1) {};

	\draw[line width=0.2mm] (5) edge node[left]{\small{\textsf{\upshape{1}}}} (1);
	\draw[line width=0.2mm] (4) edge node[left]{\small{\textsf{\upshape{1}}}} (2);
	\draw[line width=0.2mm] (3) edge node[below right=-1mm]{\small\textsf{\upshape{3}}} (2);
	\draw[line width=0.2mm] (2) edge node[below]{\small\textsf{\upshape{4}}} (1);
	\draw[line width=0.2mm, gray, dashed, rounded corners, fill=gray!10] (5a) -- (3a) -- (4a) -- cycle;

%
%
%
%
\end{tikzpicture}
\end{minipage}
\]
We then insert a new vertex and edges with multiplicity 8 as follows:
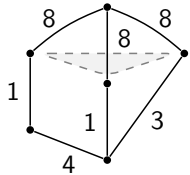
\[
\begin{minipage}{3.0cm}
\begin{tikzpicture}[vertex/.style = {circle, fill, black, minimum width = 1.mm, inner sep=0pt}, novertex/.style = {minimum width = 0.mm, inner sep=0pt}]
	\node[vertex] (0,0) (6) {};
	\node[vertex, below= 9mm of 6] (4) {};
	\node[vertex, left= 9mm of 4, yshift=4mm] (5) {};
	\node[vertex, right= 9mm of 4, yshift=4mm] (3) {};
	\node[vertex, below= 9mm of 4] (2) {};
	\node[vertex, below= 9mm of 5] (1) {};

	\coordinate[above = 0.5mm of 4] (4a) {};
	\coordinate[right = 0.5mm of 5] (5a) {};
	\coordinate[left = 0.5mm of 3] (3a) {};
	\draw[line width=0.2mm, gray, dashed, rounded corners, fill=gray!10] (5a) -- (3a) -- (4a) -- cycle;

	\draw[line width=0.2mm] (6) edge node[right, yshift=1mm]{\small\textsf{\upshape{8}}} (4);
	\draw[line width=0.2mm] (6) edge[bend left=10] node[right, yshift=1mm]{\small\textsf{\upshape{8}}} (3);
	\draw[line width=0.2mm] (6) edge[bend right=10] node[left, yshift=1mm]{\small\textsf{\upshape{8}}} (5);
	\draw[line width=0.2mm] (5) edge node[left]{\small{\textsf{\upshape{1}}}} (1);
	\draw[line width=0.2mm] (4) edge node[left]{\small{\textsf{\upshape{1}}}} (2);
	\draw[line width=0.2mm] (3) edge node[below right=-1mm]{\small\textsf{\upshape{3}}} (2);
	\draw[line width=0.2mm] (2) edge node[below]{\small\textsf{\upshape{4}}} (1);

%
%
%
%
\end{tikzpicture}
\end{minipage}
\]
Since the rank of a tensor in $\CC^n\otimes \CC^n\otimes \CC^{n^3}$ is at most $n^2$, the linear map which splits up the vertex and inserts the new vertex together with the appropriate edges with multiplicity 8 has cost at most the cost of creating the tensor corresponding to the hypergraph $\Tens(\dome_{1,4})^{\otimes 2}$ of Lemma\nobreakspace \ref {domelemma2}. Thus, $\omega(\Tens(G)) \leq 4\exponent(\tfrac14,1,1) + 2\exponent(\Tens(\dome_{1,4})) \leq 4\cdot2 + 2\cdot 12 = 32$.
\end{proof}

\subsection{Tensor surgery for hypergraphs}

In the second example we will be inserting a hypergraph into a hypergraph.
Define $H$ as the hypergraph 
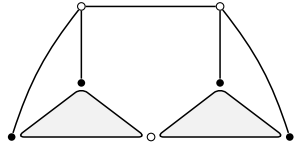
\[
\begin{tikzpicture}[vertex/.style = {circle, fill, black, minimum width = 1.mm, inner sep=0pt}, nofillvertex/.style = {draw, circle, black, minimum width = 1.mm, inner sep=0pt}, novertex/.style = {minimum width = 0.mm, inner sep=0pt}]
	\node[nofillvertex] (0,0) (X) {};
	\node[vertex, below= 9mm of X] (Y) {};
	\coordinate[novertex, below= 0.1mm of Y] (Y1);
	\node[nofillvertex, below right= 9mm of Y, xshift=2mm] (Z) {};
	\coordinate[novertex, left= 0.1mm of Z] (Z1);
	\node[vertex, below left= 9mm of Y, xshift=-2mm] (V) {};
	\coordinate[novertex, right= 0.1mm of V] (V1);

	\draw[line width=0.2mm] (X) -- node[right, xshift=-0.5mm, yshift=-1mm] {} (Y);
	\draw[line width=0.2mm] (X) edge[bend right=10] node[left] {} (V);

	\draw[rounded corners=1mm, fill=gray!10, line width=0.2mm] (Y1) -- (Z1) --  (V1) -- cycle;
	\node[nofillvertex, xshift=18.25mm] (0,0) (Xa) {};
	\node[vertex, below= 9mm of Xa] (Ya) {};
	\coordinate[novertex, below= 0.1mm of Ya] (Y1a);
	\node[vertex, below right= 9mm of Ya, xshift=2mm] (Za) {};
	\coordinate[novertex, left= 0.1mm of Za] (Z1a);
	\coordinate[novertex, right= 0.1mm of Z] (V1a);

	\draw[line width=0.2mm] (Xa) -- node[right, xshift=-0.5mm, yshift=-1mm] {} (Ya);
	\draw[line width=0.2mm] (Xa) edge[bend left=10] node[right] {} (Za);

	\draw[rounded corners=1mm, fill=gray!10, line width=0.2mm] (Y1a) -- (Z1a) -- (V1a) -- cycle;

	\draw[line width=0.2mm] (X) -- (Xa);
\end{tikzpicture}
\]

\begin{proposition}\label{prop44}
We have $6 \leq \omega(\Tens(H)) \leq 6\omega / 2$.
\end{proposition}
\begin{proof}
The lower bound follows from grouping the white vertices together and grouping the black vertices together, and taking the corresponding flattening. For the upper bound, we start off with the dome $\dome_{1,1}$
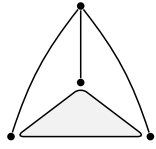
\[
\begin{tikzpicture}[vertex/.style = {circle, fill, black, minimum width = 1.mm, inner sep=0pt}, novertex/.style = {minimum width = 0.mm, inner sep=0pt}]
	\node[vertex] (0,0) (X) {};
	\node[vertex, below= 9mm of X] (Y) {};
	\coordinate[novertex, below= 0.1mm of Y] (Y1);
	\node[vertex, below right= 9mm of Y, xshift=2mm] (Z) {};
	\coordinate[novertex, left= 0.1mm of Z] (Z1);
	\node[vertex, below left= 9mm of Y, xshift=-2mm] (V) {};
	\coordinate[novertex, right= 0.1mm of V] (V1);

	\draw[line width=0.2mm] (X) -- node[right, xshift=-0.5mm, yshift=-1mm] {} (Y);
	\draw[line width=0.2mm] (X) edge[bend left=10] node[right] {} (Z);
	\draw[line width=0.2mm] (X) edge[bend right=10] node[left] {} (V);

	\draw[rounded corners=1mm, fill=gray!10, line width=0.2mm] (Y1) -- (Z1) --  (V1) -- cycle;
\end{tikzpicture}
\]
We split one of the vertices in the hyperedge, as follows
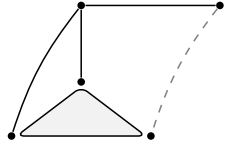
\[
\begin{tikzpicture}[vertex/.style = {circle, fill, black, minimum width = 1.mm, inner sep=0pt}, novertex/.style = {minimum width = 0.mm, inner sep=0pt}]
	\node[vertex] (0,0) (X) {};
	\node[vertex, below= 9mm of X] (Y) {};
	\coordinate[novertex, below= 0.1mm of Y] (Y1);
	\node[vertex, below right= 9mm of Y, xshift=2mm] (Z) {};
	\coordinate[novertex, left= 0.1mm of Z] (Z1);
	\node[vertex, below left= 9mm of Y, xshift=-2mm] (V) {};
	\coordinate[novertex, right= 0.1mm of V] (V1);

	\draw[line width=0.2mm] (X) -- node[right, xshift=-0.5mm, yshift=-1mm] {} (Y);
	\draw[line width=0.2mm] (X) edge[bend right=10] node[left] {} (V);

	\draw[rounded corners=1mm, fill=gray!10, line width=0.2mm] (Y1) -- (Z1) -- (V1) -- cycle;
	\node[vertex, xshift=18.25mm] (0,0) (Xa) {};
	
	\draw[line width=0.2mm, dashed, gray] (Xa) edge[bend right=10] (Z);

	\draw[line width=0.2mm] (X) -- (Xa);
\end{tikzpicture}
\]
and insert the remaining vertices and edges as to obtain the goal tensor.
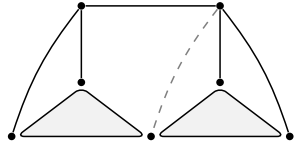
\[
\begin{tikzpicture}[vertex/.style = {circle, fill, black, minimum width = 1.mm, inner sep=0pt}, novertex/.style = {minimum width = 0.mm, inner sep=0pt}]
	\node[vertex] (0,0) (X) {};
	\node[vertex, below= 9mm of X] (Y) {};
	\coordinate[novertex, below= 0.1mm of Y] (Y1);
	\node[vertex, below right= 9mm of Y, xshift=2mm] (Z) {};
	\coordinate[novertex, left= 0.1mm of Z] (Z1);
	\node[vertex, below left= 9mm of Y, xshift=-2mm] (V) {};
	\coordinate[novertex, right= 0.1mm of V] (V1);

	\draw[line width=0.2mm] (X) -- node[right, xshift=-0.5mm, yshift=-1mm] {} (Y);
	\draw[line width=0.2mm] (X) edge[bend right=10] node[left] {} (V);

	\draw[rounded corners=1mm, fill=gray!10, line width=0.2mm] (Y1) -- (Z1) -- (V1) -- cycle;
	\node[vertex, xshift=18.25mm] (0,0) (Xa) {};
	\node[vertex, below= 9mm of Xa] (Ya) {};
	\coordinate[novertex, below= 0.1mm of Ya] (Y1a);
	\node[vertex, below right= 9mm of Ya, xshift=2mm] (Za) {};
	\coordinate[novertex, left= 0.1mm of Za] (Z1a);
	\coordinate[novertex, right= 0.1mm of Z] (V1a);

	\draw[line width=0.2mm] (Xa) -- node[right, xshift=-0.5mm, yshift=-1mm] {} (Ya);
	\draw[line width=0.2mm] (Xa) edge[bend left=10] node[right] {} (Za);

	\draw[rounded corners=1mm, fill=gray!10, line width=0.2mm] (Y1a) -- (Z1a) -- (V1a) -- cycle;

	\draw[line width=0.2mm, dashed, gray] (Xa) edge[bend right=10] (Z);
	\draw[line width=0.2mm] (X) -- (Xa);
\end{tikzpicture}
\]
We see that the combined cost of splitting the vertex and inserting the vertices and edges is at most $\omega(\Tens(\dome_{1,1}))$ which is at most $3\omega/2$ (Lemma\nobreakspace \ref {domelemma}). We conclude that the inquality $\omega(\Tens(G)) \leq 2\exponent(\Tens(\dome_{1,1})) \leq 6\omega/2$ holds.
\end{proof}

Of course, by replacing $\dome_{1,1}$ by $\dome_{1,4}$ one can obtain an exact result like in Proposition\nobreakspace \ref {dome}.

One of the reviewers observed the following.
The upper bound of Proposition\nobreakspace \ref{prop44} can also be obtained by covering by cycle 
graphs and distillation in a way similar to Lem. 4.1:
Number the vertices of $H$ left to right, top to bottom. Combine the 5-cycles
12463, 12765 and 3-cycles 135, 247, then distill 3-cycles 356 and 467 to 
obtain hyperedges. The resulting hypergraph is the hypergraph $H$ doubled.

\paragraph{Acknowledgements.} The authors thank Markus Bläser for helpful discussions. MC acknowledges financial support from the European Research Council (ERC Grant Agreement no.~337603), the Danish Council for Independent Research (Sapere Aude), and VILLUM FONDEN via the QMATH Centre of Excellence (Grant no.~10059).  JZ is supported by NWO through the research programme 617.023.116.

\raggedright
\bibliographystyle{alphaurlpp}
\bibliography{surgery_cc}                                             %

\newcommand{\etalchar}[1]{$^{#1}$}
\begin{thebibliography}{CCD{\etalchar{+}}10}

\bibitem[AFLG15]{MR3388238}
Andris Ambainis, Yuval Filmus, and Fran\c{c}ois Le~Gall.
\newblock \href {http://dx.doi.org/10.1145/2746539.2746554} {Fast matrix
  multiplication: limitations of the {C}oppersmith-{W}inograd method}.
\newblock In {\em S{TOC}'15---{P}roceedings of the 2015 {ACM} {S}ymposium on
  {T}heory of {C}omputing}, pages 585--593. ACM, New York, 2015.

\bibitem[BCS97]{burgisser1997algebraic}
Peter B{\"u}rgisser, Michael Clausen, and M.~Amin Shokrollahi.
\newblock \href {http://dx.doi.org/10.1007/978-3-662-03338-8} {{\em Algebraic
  complexity theory}}, volume 315 of {\em Grundlehren Math. Wiss.}
\newblock Springer-Verlag, Berlin, 1997.

\bibitem[BCZ17]{buhrman2016nondeterministic}
Harry Buhrman, Matthias Christandl, and Jeroen Zuiddam.
\newblock \href {http://dx.doi.org/10.4230/LIPIcs.ITCS.2017.24}
  {{Nondeterministic Quantum Communication Complexity: the Cyclic Equality Game
  and Iterated Matrix Multiplication}}.
\newblock In {\em 8th Innovations in Theoretical Computer Science Conference
  (ITCS 2017)}, pages 24:1--24:18, 2017.

\bibitem[BI11]{MR2932001}
Peter B\"urgisser and Christian Ikenmeyer.
\newblock \href {http://dx.doi.org/10.1145/1993636.1993704} {Geometric
  complexity theory and tensor rank}.
\newblock In {\em S{TOC}'11---{P}roceedings of the 43rd {ACM} {S}ymposium on
  {T}heory of {C}omputing}, pages 509--518. ACM, New York, 2011.

\bibitem[Bl{\"{a}}03]{DBLP:journals/jc/Blaser03}
Markus Bl{\"{a}}ser.
\newblock \href {http://dx.doi.org/10.1016/S0885-064X(02)00007-9} {On the
  complexity of the multiplication of matrices of small formats}.
\newblock {\em J. Complexity}, 19(1):43--60, 2003.

\bibitem[CCD{\etalchar{+}}10]{chen2010tensor}
Lin Chen, Eric Chitambar, Runyao Duan, Zhengfeng Ji, and Andreas Winter.
\newblock \href {https://doi.org/10.1103/PhysRevLett.105.200501} {Tensor rank
  and stochastic entanglement catalysis for multipartite pure states}.
\newblock {\em Phys. Rev. Lett.}, 105(20):200501, 2010.

\bibitem[CM96]{comon1996decomposition}
Pierre Comon and Bernard Mourrain.
\newblock \href {http://dx.doi.org/10.1016/0165-1684(96)00079-5} {Decomposition
  of quantics in sums of powers of linear forms}.
\newblock {\em Signal Processing}, 53(2):93--107, 1996.

\bibitem[CU12]{MR3202968}
Henry Cohn and Christopher Umans.
\newblock \href {http://dx.doi.org/10.1137/1.9781611973105.77} {Fast matrix
  multiplication using coherent configurations}.
\newblock In {\em Proceedings of the {T}wenty-{F}ourth {A}nnual {ACM}-{SIAM}
  {S}ymposium on {D}iscrete {A}lgorithms}, pages 1074--1087. SIAM,
  Philadelphia, PA, 2012.

\bibitem[CVZ16]{christandl2016asymptotic}
Matthias Christandl, P{\'e}ter Vrana, and Jeroen Zuiddam.
\newblock \href {https://arxiv.org/abs/1609.07476} {Asymptotic tensor rank of
  graph tensors: beyond matrix multiplication}.
\newblock {\em arXiv}, 2016.
\newblock \href {http://arxiv.org/abs/1609.07476} {\path{arXiv:1609.07476}}.

\bibitem[CW90]{coppersmith1987matrix}
Don Coppersmith and Shmuel Winograd.
\newblock \href {http://dx.doi.org/10.1016/S0747-7171(08)80013-2} {Matrix
  multiplication via arithmetic progressions}.
\newblock {\em J.~Symbolic Comput.}, 9(3):251--280, 1990.

\bibitem[dG78]{de1978varieties}
Hans~F. de~Groote.
\newblock \href {http://dx.doi.org/10.1016/0304-3975(78)90038-5} {On varieties
  of optimal algorithms for the computation of bilinear mappings. {I}. {T}he
  isotropy group of a bilinear mapping}.
\newblock {\em Theoret. Comput. Sci.}, 7(1):1--24, 1978.

\bibitem[DKW11]{draisma2009partition}
Jan Draisma, Eyal Kushilevitz, and Enav Weinreb.
\newblock \href {http://dx.doi.org/10.1016/j.tcs.2010.01.018} {Partition
  arguments in multiparty communication complexity}.
\newblock {\em Theoret. Comput. Sci.}, 412(24):2611--2622, 2011.

\bibitem[H{\aa}s90]{haastad1990tensor}
Johan H{\aa}stad.
\newblock \href {http://dx.doi.org/10.1016/0196-6774(90)90014-6} {Tensor rank
  is {NP}-complete}.
\newblock {\em J. Algorithms}, 11(4):644--654, 1990.

\bibitem[HK71]{hopcroft1971minimizing}
John~E. Hopcroft and Leslie~R. Kerr.
\newblock \href {http://dx.doi.org/10.1137/0120004} {On minimizing the number
  of multiplications necessary for matrix multiplication}.
\newblock {\em SIAM J. Appl. Math.}, 20:30--36, 1971.

\bibitem[Lan06]{landsberg2006border}
Joseph~M. Landsberg.
\newblock \href {http://dx.doi.org/10.1090/S0894-0347-05-00506-0} {The border
  rank of the multiplication of {$2\times2$} matrices is seven}.
\newblock {\em J. Amer. Math. Soc.}, 19(2):447--459, 2006.

\bibitem[Lan12]{landsberg2012tensors}
Joseph~M. Landsberg.
\newblock {\em Tensors: geometry and applications}, volume 128 of {\em Graduate
  Studies in Mathematics}.
\newblock American Mathematical Society, Providence, RI, 2012.

\bibitem[LG12]{le2012faster}
Fran\c{c}ois Le~Gall.
\newblock \href {http://dx.doi.org/10.1109/FOCS.2012.80} {Faster algorithms for
  rectangular matrix multiplication}.
\newblock In {\em 2012 {IEEE} 53rd {A}nnual {S}ymposium on {F}oundations of
  {C}omputer {S}cience---{FOCS} 2012}, pages 514--523. IEEE Computer Soc., Los
  Alamitos, CA, 2012.

\bibitem[LG14]{le2014powers}
Fran{\c{c}}ois Le~Gall.
\newblock \href {http://dx.doi.org/10.1145/2608628.2608664} {Powers of tensors
  and fast matrix multiplication}.
\newblock In {\em I{SSAC} 2014---{P}roceedings of the 39th {I}nternational
  {S}ymposium on {S}ymbolic and {A}lgebraic {C}omputation}, pages 296--303.
  ACM, New York, 2014.

\bibitem[LO13]{MR3081636}
Joseph~M. Landsberg and Giorgio Ottaviani.
\newblock \href {http://dx.doi.org/10.1007/s10231-011-0238-6} {Equations for
  secant varieties of {V}eronese and other varieties}.
\newblock {\em Ann. Mat. Pura Appl. (4)}, 192(4):569--606, 2013.

\bibitem[LO15]{landsberg2011new}
Joseph~M. Landsberg and Giorgio Ottaviani.
\newblock \href {http://dx.doi.org/10.4086/toc.2015.v011a011} {New lower bounds
  for the border rank of matrix multiplication}.
\newblock {\em Theory Comput.}, 11:285--298, 2015.

\bibitem[PS05]{pachter2005algebraic}
Lior Pachter and Bernd Sturmfels, editors.
\newblock \href
  {https://www.cambridge.org/core/books/algebraic-statistics-for-computational-biology/2E5CCE6BB6751EB7423EE3D2BF40EBFF}
  {{\em Algebraic Statistics for Computational Biology}}.
\newblock Cambridge University Press, Aug 2005.

\bibitem[Shi16]{shitov2016hard}
Yaroslav Shitov.
\newblock How hard is the tensor rank?
\newblock {\em arXiv preprint arXiv:1611.01559}, 2016.

\bibitem[Smi13]{smirnov2013bilinear}
A.~V. Smirnov.
\newblock \href {http://dx.doi.org/10.1134/S0965542513120129} {The bilinear
  complexity and practical algorithms for matrix multiplication}.
\newblock {\em Zh. Vychisl. Mat. Mat. Fiz.}, 53(12):1970--1984, 2013.

\bibitem[SS16]{schaefer2016complexity}
Marcus Schaefer and Daniel Stefankovic.
\newblock The Complexity of Tensor Rank.
\newblock {\em arXiv preprint arXiv:1612.04338}, 2016.

\bibitem[Str69]{strassen1969gaussian}
Volker Strassen.
\newblock \href {http://dx.doi.org/10.1007/BF02165411} {Gaussian elimination is
  not optimal}.
\newblock {\em Numer. Math.}, 13(4):354--356, 1969.

\bibitem[Str83]{MR709378}
Volker Strassen.
\newblock \href {http://dx.doi.org/10.1016/0024-3795(83)80041-X} {Rank and
  optimal computation of generic tensors}.
\newblock {\em Linear Algebra Appl.}, 52/53:645--685, 1983.

\bibitem[Str87]{strassen1987relative}
Volker Strassen.
\newblock \href {http://dx.doi.org/10.1515/crll.1987.375-376.406} {Relative
  bilinear complexity and matrix multiplication}.
\newblock {\em J. Reine Angew. Math.}, 375/376:406--443, 1987.

\bibitem[Str88]{MR929980}
Volker Strassen.
\newblock \href {http://dx.doi.org/10.1515/crll.1988.384.102} {The asymptotic
  spectrum of tensors}.
\newblock {\em J. Reine Angew. Math.}, 384:102--152, 1988.

\bibitem[Str91]{MR1089800}
Volker Strassen.
\newblock \href {http://dx.doi.org/10.1515/crll.1991.413.127} {Degeneration and
  complexity of bilinear maps: some asymptotic spectra}.
\newblock {\em J. Reine Angew. Math.}, 413:127--180, 1991.

\bibitem[Str05]{MR2138544}
Volker Strassen.
\newblock Komplexit\"at und {G}eometrie bilinearer {A}bbildungen.
\newblock {\em Jahresber. Deutsch. Math.-Verein.}, 107(1):3--31, 2005.

\bibitem[VC15]{vrana2015asymptotic}
P\'eter Vrana and Matthias Christandl.
\newblock \href {http://dx.doi.org/10.1063/1.4908106} {Asymptotic entanglement
  transformation between {W} and {GHZ} states}.
\newblock {\em J. Math. Phys.}, 56(2):022204, 12, 2015.
\newblock \href {http://arxiv.org/abs/1310.3244} {\path{arXiv:1310.3244}}.

\bibitem[VC17]{vrana}
P{\'e}ter Vrana and Matthias Christandl.
\newblock \href {http://dx.doi.org/10.1007/s00220-017-2861-6} {Entanglement
  Distillation from Greenberger--Horne--Zeilinger Shares}.
\newblock {\em Commun. Math. Phys.}, 352(2):621--627, 2017.
\newblock \href {http://arxiv.org/abs/1603.03964} {\path{arXiv:1603.03964}}.

\bibitem[Win71]{winograd1971multiplication}
Shmuel Winograd.
\newblock \href {http://dx.doi.org/10.1016/0024-3795(71)90009-7} {On
  multiplication of 2$\times$2 matrices}.
\newblock {\em Linear Algebra Appl.}, 4(4):381--388, 1971.

\end{thebibliography}

%
%
%
%
%
%

\end{document}